\documentclass[a4paper,11pt]{amsart}

\usepackage{verbatim}
\usepackage[dvips]{color}
\usepackage{amssymb}
\usepackage{amscd}
\usepackage{amsfonts,esint}
\usepackage{indentfirst}
\usepackage{verbatim}
\usepackage{amsmath}
\usepackage{amsthm}
\usepackage{enumerate}
\usepackage{graphicx}
\usepackage{subfig}
\usepackage{xcolor}
\usepackage{amssymb}
\usepackage{bbm}

\def\dashint{\fint}

\usepackage[active]{srcltx}
\usepackage[colorlinks,pdfpagelabels,pdfstartview = FitH,bookmarksopen = true,bookmarksnumbered = true,linkcolor = blue,plainpages = false,hypertexnames = false,citecolor = red] {hyperref}

\allowdisplaybreaks

\title{Potential theory for nonlocal drift-diffusion equations}

\author{Quoc-Hung Nguyen}
\address{Academy of Mathematics and Systems Science,
	Chinese Academy of Sciences,
	Beijing 100190, PR China.}
\email{qhnguyen@amss.ac.cn}
\author{Simon Nowak}
\address{Fakult\"at f\"ur Mathematik, Universit\"at Bielefeld, Postfach 100131, D-33501 Bielefeld, Germany}
\email{simon.nowak@uni-bielefeld.de}
 \author{Yannick Sire}
 \address{Department of Mathematics, Johns Hopkins University, Baltimore, MD, 21218}
\email{ysire1@jhu.edu}
\author{Marvin Weidner} 
\address{Departament de Matemàtiques i Informàtica, Universitat de Barcelona, Gran Via de les Corts Catalanes 585, 08007 Barcelona, Spain}
\email{mweidner@ub.edu}

\newtheorem{theorem}{Theorem}[section]

\newtheorem{proposition}{Proposition}[section]
\newtheorem{lemma}{Lemma}[section]

\newtheorem{corollary}{Corollary}[section]
\theoremstyle{definition}
\newtheorem{definition}{Definition}
\newtheorem{remark}{Remark}[section]

\numberwithin{equation}{section}

\def\eqn#1$$#2$${\begin{equation}\label#1#2\end{equation}}
\def\charfn_#1{{\raise1.2pt\hbox{$\chi_{\kern-1pt\lower3pt\hbox{{$\scriptstyle#1$}}}$}}}

\newcommand{\eps}{\varepsilon}

 \DeclareMathOperator*{\osc}{osc}

\DeclareMathOperator{\tail}{Tail}

\def\dive{\operatorname{div}}

\newcommand{\divo}{\textnormal{div}}

\usepackage{dsfont}

\newcommand{\loc}{_{loc}}

\newcommand{\cE}{\mathcal{E}}

\newcommand{\R}{\mathbb{R}}
\renewcommand{\d}{\textnormal{d}}

\def\loc{\operatorname{loc}}

\def\mean#1{\mathchoice%
          {\mathop{\kern 0.2em\vrule width 0.6em height 0.69678ex depth -0.58065ex
                  \kern -0.8em \intop}\nolimits_{\kern -0.4em#1}}%
          {\mathop{\kern 0.1em\vrule width 0.5em height 0.69678ex depth -0.60387ex
                  \kern -0.6em \intop}\nolimits_{#1}}%
          {\mathop{\kern 0.1em\vrule width 0.5em height 0.69678ex
              depth -0.60387ex
                  \kern -0.6em \intop}\nolimits_{#1}}%
          {\mathop{\kern 0.1em\vrule width 0.5em height 0.69678ex depth -0.60387ex
                  \kern -0.6em \intop}\nolimits_{#1}}}

\def\vintslides_#1{\mathchoice%
          {\mathop{\kern 0.1em\vrule width 0.5em height 0.697ex depth -0.581ex
                  \kern -0.6em \intop}\nolimits_{\kern -0.4em#1}}%
          {\mathop{\kern 0.1em\vrule width 0.3em height 0.697ex depth -0.604ex
                  \kern -0.4em \intop}\nolimits_{#1}}%
          {\mathop{\kern 0.1em\vrule width 0.3em height 0.697ex depth -0.604ex
                  \kern -0.4em \intop}\nolimits_{#1}}%
          {\mathop{\kern 0.1em\vrule width 0.3em height 0.697ex depth -0.604ex
                  \kern -0.4em \intop}\nolimits_{#1}}}

\newcommand{\aveint}[2]{\mathchoice%
          {\mathop{\kern 0.2em\vrule width 0.6em height 0.69678ex depth -0.58065ex
                  \kern -0.8em \intop}\nolimits_{\kern -0.45em#1}^{#2}}%
          {\mathop{\kern 0.1em\vrule width 0.5em height 0.69678ex depth -0.60387ex
                  \kern -0.6em \intop}\nolimits_{#1}^{#2}}%
          {\mathop{\kern 0.1em\vrule width 0.5em height 0.69678ex depth -0.60387ex
                  \kern -0.6em \intop}\nolimits_{#1}^{#2}}%
          {\mathop{\kern 0.1em\vrule width 0.5em height 0.69678ex depth -0.60387ex
                  \kern -0.6em \intop}\nolimits_{#1}^{#2}}}
                  
\newtoks\by
\newtoks\paper
\newtoks\book
\newtoks\jour
\newtoks\yr
\newtoks\pages
\newtoks\vol
\newtoks\publ

\def\name[#1, #2]{#1 #2}
\def\ota{{\hbox{\bf ???}}}
\def\cLear{\by=\ota\paper=\ota\book=\ota\jour=\ota\yr=\ota
\pages=\ota\vol=\ota\publ=\ota}
\def\endpaper{\the\by, \textit{\the\paper},
{\the\jour} \textbf{\the\vol} (\the\yr), \the\pages.\cLear}
\def\endbook{\the\by, \textit{\the\book},
\the\publ, \the\yr.\cLear}
\def\endpap{\the\by, \textit{\the\paper}, \the\jour.\cLear}
\def\endproc{\the\by, \textit{\the\paper}, \the\book, \the\publ,
\the\yr, \the\pages.\cLear}

\begin{document}

\maketitle

\begin{abstract}
The purpose of this paper is to prove new fine regularity results for nonlocal drift-diffusion equations via pointwise potential estimates. Our analysis requires only minimal assumptions on the divergence free drift term, enabling us to include drifts of critical order belonging merely to BMO. In particular, our results allow to derive new estimates for the dissipative surface quasi-geostrophic equation.
\end{abstract}

\tableofcontents

\section{Introduction}	
The goal of this paper is to derive new regularity results for drift-diffusion equations with a critical dissipative term in $\mathbb R^d$. More precisely, we establish fine estimates for solutions to such equations by means of parabolic Riesz potentials that are independent of the vector field in the drift term, except for its {\sl a priori} regularity (see \cite{kiselev} for a survey and references therein). The framework of this article is very general and, in particular, it yields new estimates for the critical dissipative SQG equation. Such equation has been introduced in \cite{CMT} as a toy model for the regularity of Navier-Stokes and several seminal works in the last years have led to major breakthroughs on this topic, see \cite{CaffaVasseur,KNV,KN}. Moreover, let us point out that some of the fine regularity estimates obtained in this article are even new for the standard fractional heat equation.

For $d \geq 2$, we consider the following general drift-diffusion equation
\begin{align}\label{eq1}
&\partial_t u + (b,\nabla u)+\mathcal{L}_t u = \mu ~~\text{in}~~I \times \Omega \subset \mathbb{R} \times \mathbb{R}^d =\mathbb{R}^{d+1},
\end{align}
where $\mu$ belongs to the set of finite measures in $\mathbb R^{d+1}$, which we denote by $\mathcal M (\mathbb R^{d+1})$, while $b : \mathbb{R} \times \mathbb{R}^d \to \mathbb{R}^d$ is a vector field that belongs to {either $L^\infty(\mathbb{R}^{d+1})$ or $L^\infty(\mathbb{R};\text{BMO}(\mathbb R^{d}))$}, and satisfies $\operatorname{div}(b(t))=0$ almost everywhere in time $t \in \R$. \\\\

The integral operator $\mathcal{L}_t$ ($t \in \mathbb{R}$) is defined by  
\begin{align}
\label{eq:L}
\mathcal{L}_t  u(x) = \text{P.V.}\int_{\mathbb{R}^d}(u(x)-u(y)){K}_t(x,y) \d y,
\end{align}
where the kernel $ {K}_t: \mathbb{R} \times \mathbb{R}^d \times \mathbb{R}^d \to [0, \infty), ~~ (t,x, y) \mapsto{K}_t(x, y) $ 
is assumed to satisfy the following conditions for a given $s \in [1/2,1)$:
\begin{itemize}
	\item[(i)] ${K}_t$ is a measurable function;
	\item[(ii)] for every $t \in \mathbb{R}$, ${K}_t$ is symmetric which means that
	\[
	{K}_t(x, y) = {K}_t(y, x), \quad \text{for a.e.} ~ (t, x, y) \in \mathbb{R} \times \mathbb{R}^d \times \mathbb{R}^d;
	\]
	\item[(iii)] there exists a constant $\Lambda \ge 1$ such that
	\[
	\Lambda^{-1} \leq |x - y|^{d + 2s} {K}_t(x, y) \leq \Lambda, ~~ \text{for a.e.} ~ (t, x, y) \in \mathbb{R} \times \mathbb{R}^d \times \mathbb{R}^d.
	\]
\end{itemize} 

An example of an operator $\mathcal L_t$ is of course given by the Fourier multiplier $(-\Delta)^s$. In this case equation \eqref{eq1} refers to a standard class of drift-diffusion equations. In the case $d=2$, $s=\frac12$, when $b$ is given by a rotation of the vectorial Riesz transform $b = \nabla^\perp (-\Delta)^{-\frac12} u$, one obtains the dissipative surface quasi-geostrophic equation (SQG equation). We note at this point that in the case $s=\frac12$, the drift term has the same order as the diffusion term given by the nonlocal operator in \eqref{eq1}, so that the problem becomes \emph{critical} instead of subcritical, leading to substantial additional technical difficulties compared to parabolic nonlocal equations without drift.

The main results of the present paper establish potential-theoretic estimates of (weak) solutions $u$ to \eqref{eq1} in terms of parabolic Riesz potentials of the measure $\mu$ (see \autoref{thm:PE} and \autoref{BMOdata}). In this context, we extend similar results previously obtained in the nonlocal elliptic setting (see \cite{KMS1,KMS2}) to the setting of time-dependent drift-diffusion equations, including in particular the dissipative SQG equation.
Moreover, thanks to the known mapping properties of the potentials, this framework provides new estimates in terms of finer scales of function spaces. This philosophy of obtaining fine regularity estimates via potentials follows a by now long tradition started in the context of local elliptic and parabolic equations, see e.g.\ \cite{KM94,TWAJM,Min,duzaarMin,Cianchi,KuMi,KuMiV,BCDKS,BYP,DongJEMS,DFJMPA}. \\
Let us now explain and discuss our main results in detail. We consider the cases $b \in L^{\infty}(\R^{d+1})$ and $b \in L^{\infty}(\R ; \text{BMO}(\R^d))$ separately.

\subsection*{Main results for bounded drifts}

A crucial step in the construction of potential estimates is to establish suitably localized H\"older estimates for weak solutions to the homogeneous problem corresponding to \eqref{eq1} in parabolic cylinders. In our case, the homogeneous problem is given by 
\begin{align}
\label{eq:PDE-hom}
\partial_t v + (b,\nabla v) + \mathcal L_t v = 0 ~~ \text{ in } Q_{r}(t_0,x_0),
\end{align}
where $Q_r(t_0,x_0):= I_r^{\ominus}(t_0) \times B_r(x_0)$ denotes the parabolic cylinder and $I_r^{\ominus}(t_0) := (t_0-r^{2s},t_0)$.

Weak solutions to the homogeneous problem \eqref{eq:PDE-hom} and the inhomogeneous problem \eqref{eq1} are defined as follows:

\begin{definition}[Weak solutions] Let $\mu \in \mathcal{M}(\mathbb{R}^{d+1})$. A function \newline $u\in L^\infty(I;L^2(\Omega))\cap L^2(I;H^{s}(\Omega)) \cap L^2(I;L^1_{2s}(\mathbb{R}^d))$ is a weak subsolution to \eqref{eq1}, if 
	for all $(t_1,t_2) \subset I$ and nonnegative test functions $\varphi\in L^2((t_1,t_2); H^s(\R^d)) \cap C^1((t_1,t_2);L^2(\Omega))$ with $\varphi \equiv 0$ in $(t_1,t_2) \times \R^d \setminus \Omega$
	we have
	\begin{align*}
		& \quad -\int_{t_1}^{t_2} \int_{\Omega} u\left[\partial_t\varphi + (b,\nabla \varphi) \right]\d x \d t + \int_{\Omega} u(t_2,x)\varphi(t_2,x) \d x - \int_{\Omega} u(t_1,x)\varphi(t_1,x) \d x \\ & \quad + \int_{t_1}^{t_2} \int_{\mathbb{R}^d} \int_{\mathbb{R}^d} (u(t,x)-u(t,y))(\varphi(t,x)-\varphi(t,y)) K_t(x,y) \d y \d x \d t \\& \le \int_{t_1}^{t_2} \int_{\Omega} \varphi \d\mu .
	\end{align*}
	Here, $L^1_{2s}(\R^d)$ denotes the nonlocal tail space
	$$L^1_{2s}(\mathbb{R}^d):= \left \{v \in L^1_{\loc}(\mathbb{R}^d) \mathrel{\Big|} \int_{\mathbb{R}^d} \frac{|v(y)|}{(1+|y|)^{d+2s}} \d y < \infty \right \}.$$
A function $u$ is called a weak supersolution if the previous estimate holds true for every nonpositive test function $\phi$, and it is called a weak solution if it is a weak subsolution and a weak supersolution.\\
$u$ is called a weak (sub/super)solution to \eqref{eq:PDE-hom} if it is a weak (sub/super)solution to \eqref{eq1} with $\mu \equiv 0$.
\end{definition}


Next, we state the localized H\"older regularity estimate for solutions to the homogeneous problem \eqref{eq:PDE-hom} under the assumption that the vector field $b$ is bounded. 

\begin{theorem}[H\"older estimate]
	\label{prop:Holderq}
	Let $s \in [1/2,1)$ and $b \in L^\infty(\mathbb R^{d+1})$ with $\Vert b \Vert_{L^{\infty}(\R^{d+1})} \le \Lambda$. Suppose that $v$ is a weak solution to \eqref{eq:PDE-hom} in a parabolic cylinder $Q_{r}(t_0,x_0)$ for some $r > 0$. Then, there exist $\alpha \in (0,1)$ and $c > 0$ depending only on $d,s,\Lambda,q$ such that $v \in C^{\alpha}_{loc}(Q_{r}(t_0,x_0))$ and the following estimate holds true for any $q > 1$:
	\begin{align*}
		& \Vert v \Vert_{C^{\alpha}(Q_{r/2}(t_0,x_0))} \\ & \quad \leq c r^{-\alpha}  \left[\left(\dashint_{Q_{r}(t_0,x_0)} |v| \d x \d t \right) +\left(\dashint_{I_r^{\ominus}(t_0)} \tail(v(t);x_0,r)^q \d t \right)^{1/q} \right],
	\end{align*}
	where the nonlocal tail term is defined as
	\begin{align*}
\tail(v(t);x_0,r) = r^{2s} \int_{\R^d \setminus B_r(x_0)} |v(t,y)||x_0-y|^{-d-2s} \d y.
\end{align*}
\end{theorem}

The previous result is already new in contrast to previous H\"older estimates for nonlocal drift-diffusion equations (see e.g.\ \cite{CaffaVasseur}, \cite{sireDCDS}, \cite{delgadinoSmith}) because of two reasons. First, our estimate is truly local in nature in the sense that it does not require $v$ to solve an equation in the full space. Second, our estimate of the H\"older norm only contains tail terms which are merely required to belong to $L^q$ in time for some $q>1$. Since this quantity is always finite for any weak solution, our result does not impose any artificial restriction on the behavior of solutions at infinity. Note that Theorem \ref{prop:Holderq} does not hold with $q=1$ due to \cite[Example 5.2]{KaWe23}.

In comparison to previous results on H\"older regularity for nonlocal diffusion equations, our result can be regarded as an extension of \cite{KaWe23} since it allows for subcritical and critical drift terms. For previous results on nonlocal elliptic and parabolic equations without drift see for instance \cite{KassCalcVar,CCV,FelsKass,DKP,Coz17,KaWe22a}. Finally, let us mention the articles \cite{SilvestreAdv, SilvestrePisa, SilvestreIndiana}, where H\"older estimates have been obtained for nonlocal drift-diffusion equations in case $\mathcal{L}_t = (-\Delta)^s$ using non-variational techniques.\\

We proceed by introducing parabolic Riesz potentials:

\begin{definition}
Let $\mu \in \mathcal M(\mathbb R^{d+1})$ be a finite Radon measure with finite total mass and consider a parabolic cylinder $Q_\rho(t_0,x_0)=(t_0-\rho^{2s},t_0) \times B_\rho(x_0)$. Then we denote by $\mathcal {P}^R_{a}[\mu]$ the parabolic Riesz potential of $\mu$ with order $a\in (0,d+2s)$ which is defined as follows
\begin{align*}
\mathcal {P}^R_{a}[\mu](t_0,x_0)=\int_{0}^{R}\frac{|\mu|(Q_\rho(t_0,x_0))}{\rho^{d+2s-a}}\frac{d\rho}{\rho}. 
\end{align*}
In particular, one has 
\begin{align*}
	\mathcal {P}^R_{2s}[\mu](t_0,x_0)=\int_{0}^{R}\frac{|\mu|(Q_\rho(t_0,x_0))}{\rho^{d}}\frac{d\rho}{\rho}. 
\end{align*}
\end{definition}

\begin{remark}
Note that while the order $2s$ of the parabolic potential does not explicitly appear in its definition, it is intrinsically encoded via the size of the parabolic cylinder $Q_\rho(t_0,x_0)$.
\end{remark}

We are now in position to state our first main result, a zero-order potential estimate for solutions to \eqref{eq1} under the assumption that the drift $b$ is bounded uniformly in space and time, which in particular includes solutions to the SQG equation.  

\begin{theorem}[Potential estimate] \label{thm:PE}
	Let $s \in [1/2,1)$, $\mu \in \mathcal{M}(\mathbb{R}^{d+1})$,  assume that $b \in L^\infty(\mathbb R^{d+1})$ with $\Vert b \Vert_{L^{\infty}(\R^{d+1})} \le \Lambda$, and that $u$ is a weak solution of \eqref{eq1} in $I \times \Omega \subset \mathbb{R}^{d+1}$.    Then for almost all $t_0 \in I$, $x_0 \in \Omega$, any $q>1$, and any $R >0$ such that $Q_R(t_0,x_0) \subset I \times \Omega$, we have the estimate
	\begin{equation} \label{eq:PE}
		\begin{aligned}
			|u(t_0,x_0)| & \leq c \Bigg [\left(\dashint_{Q_{R}(t_0,x_0)} |u|^q \d x \d t \right)^{1/q} +\left(\dashint_{I_R^{\ominus}(t_0)} \tail(u(t);x_0,R)^q \d t \right)^{1/q} \\ & \quad + \mathcal {P}^{R}_{2s}[\mu](t_0,x_0) \Bigg ],
		\end{aligned}
	\end{equation}
	where $c$ depends only on $d,s,\Lambda,q$.
\end{theorem}

\begin{remark}[Equations without drift]
\label{remark:without-drift}
In case $b=0$, that is, when \eqref{eq1} is a parabolic nonlocal diffusion equation without drift, the potential estimate \autoref{thm:PE} remains valid for the full range $s \in (0,1)$ without any changes in the proof (see also Remark \ref{remark:proof-without-drift}).
\end{remark}

Our main result \autoref{thm:PE} extends corresponding potential estimates in the local parabolic case (see e.g.\ \cite{duzaarMin}) and in the nonlocal elliptic case (see e.g.\ \cite{KMS1,KMS2,KNS,DNCZ}) to the setting of nonlocal drift-diffusion equations. Moreover, it seems to be new already for the fractional heat equation with $\mathcal{L}_t = (-\Delta)^s$, $s \in (0,1)$, and $b \equiv 0$ (see Remark \ref{remark:without-drift}).

\begin{remark}[SOLA]
	While for the sake of simplicity we work with the above notion of energy-type weak solutions, in view of standard approximation arguments we expect our main results to remain valid for distributional solutions which can be approximated by weak solutions to regularized equations in a suitable fashion. Solutions of this type are often called SOLA (=\hspace{0.1mm} solutions obtained by limiting approximations) and can usually be shown to exist under general measure data. See for instance \cite{KMS1,KNS} for a precise solution concept of this type in the nonlocal elliptic setting and \cite{BDG,duzaarMin,KuMi} for corresponding ones in the local parabolic setting.
\end{remark}

Our main technical tool for proving Theorem \ref{thm:PE} is the following excess-decay lemma, which in particular encodes the zero-order regularity properties
of solutions to \eqref{eq1} in a precise way.

\begin{lemma}[Excess decay] \label{lemma:OscDecIntro} 
	Let $s \in [1/2,1)$, $R>0$, $t_0 \in I$, $x_0 \in \mathbb{R}^n$, $q>1$, $\mu \in \mathcal{M}(\mathbb{R}^{d+1})$, and assume that 
	$b \in L^\infty(\mathbb R^{d+1})$ with $\Vert b \Vert_{L^{\infty}(\R^{d+1})} \le \Lambda$. In addition, let $\alpha \in (0,s(1-1/q))$ be given by Theorem \ref{prop:Holderq}. Let $u$ be a weak solution to \eqref{eq1} in $I \times \Omega$ and assume that $Q_{2R}(t_0,x_0) \subset I \times \Omega$.
	Then for any integer $m \geq 1$, we have
	\begin{equation} \label{eq:oscdecay}
		\begin{aligned}
			E(u,t_0,x_0,2^{-m} R) & \leq C_0 2^{-\alpha m} E(u,t_0,x_0,R) \\ & \quad + C_0 2^{((d+2s)/q)m} R^{-d} |\mu|(Q_{R}(t_0,x_0)),
		\end{aligned}
	\end{equation}
	where $C_0 \geq 1$ depends only on $d,s,\Lambda,\alpha,q$ and for any $r>0$,
	\begin{align*}
		E(u,t_0,x_0,r)& := \left (\dashint_{I_{r}^{\ominus}(t_0)} \left [ \dashint_{B_{r}(x_0)} |u - (u)_{Q_{r}(t_0,x_0)} | \d x \right ]^q \d t \right )^{1/q} \\ & \quad+ r^{-2s/q} ||\tail(u(\cdot)-(u)_{Q_{r}(t_0,x_0)};x_0,r)||_{L^q(I_r^{\ominus}(t_0))} .
	\end{align*}
\end{lemma}

The proof of the previous parabolic potential estimate is nonlinear in nature, not because the equation \eqref{eq1} is a nonlinear PDE, but as for the techniques we use, which are reminiscent of the works \cite{duzaarMin,KMS2}. In particular, we believe that our approach has the potential to be extended to obtaining similar results for nonlinear generalizations of the equation \eqref{eq1}. \\

Let us end this section by providing an alternative perspective. In this paper, we also derive upper estimates for the fundamental solution of the operator $\partial_t + (b,\nabla \cdot) +\mathcal L_t$ (heat kernel). Such estimates could in principle also be used to derive zero-order potential estimates for solutions of \eqref{eq1}. 

\begin{theorem}[Upper heat kernel estimate]
	\label{thm:uhkeintro}
	Let $s \in [1/2,1)$ and assume that 
	$b \in L^\infty(\mathbb R^{d+1})$ with $\Vert b \Vert_{L^{\infty}(\R^{d+1})} \le \Lambda$. Let $0 \le \eta \le t \le T$ and let $p$ be the heat kernel associated to the operator $\partial_t + (b,\nabla \cdot) +\mathcal L_t$ defined on $\mathbb R^d$. Then the following holds
	\begin{align}
		\label{eq:uhkeintro}
		p(\eta,x;t,y) \le c \frac{t-\eta}{\left(|x-y|+|t-\eta|^{\frac{1}{2s}}\right)^{d+2s}} \quad \forall x,y \in \R^d
	\end{align}
	for some constant $c > 0$, depending only on $d,s,\Lambda,T$.
\end{theorem}

Our proof of \autoref{thm:uhkeintro} is based on linear techniques and follows a strategy going back to the celebrated work by Aronson (see \cite{Aro68}), which has recently been extended to a nonlocal setting by the last-named author (see \cite{KaWe23b}). We note that similar upper bounds for the heat kernel of drift-diffusion equations have been proved in \cite{MaMi13a}, \cite{MaMi13b} by using an adaptation of the celebrated Davies' method (see \cite{Dav87}, \cite{Dav89}), which was first applied to nonlocal problems in \cite{CKS87}.

\subsection*{Fine space-time regularity}

In the following, we mention some consequences of our nonlocal potential estimate \autoref{thm:PE}. Indeed, it is well-known (see e.g. \cite{duzaarMin,QH} and the references therein) that potential estimates lead to space-time estimates for solutions, using the mapping properties of the parabolic Riesz potential in Lorentz/Marcinkiewicz spaces. More precisely, 
we have the following fine regularity results. 

\begin{corollary}[Lorentz regularity] \label{cor:Lorentzreg}
	Let $s \in [1/2,1)$, $\mu \in \mathcal{M}(\mathbb{R}^{d+1})$,
	assume that $b \in L^\infty(\mathbb R^{d+1})$, and that $u$ is a weak solution of \eqref{eq1} in $I \times \Omega \subset \mathbb{R}^{d+1}$.
	\begin{itemize}
		\item For any $p \in \left (1,\frac{d+2s}{2s} \right )$ and any $\sigma \in (0,\infty]$, we have the implication 
		\begin{equation} \label{Lorentz2}
			\mu \in L^{p,\sigma}(I \times \Omega) \implies u \in L^{\frac{p(d+2s)}{d+2s-2sp},\sigma}_{loc}(I \times \Omega).
		\end{equation}
		\item We have
		\begin{equation} \label{Lorentz3}
			\mu \in L^{\frac{d+2s}{2s},1}(I \times \Omega) \implies u \in L^{\infty}_{loc}(I \times \Omega).
		\end{equation}
	\end{itemize}
\end{corollary}

\begin{remark}[Equations without drift]
\label{remark:without-driftx}
	In case $b=0$, Corollary \ref{cor:Lorentzreg} remains valid for the full range $s \in (0,1)$ by using the corresponding potential estimate that was discussed in Remark \ref{remark:without-drift}.
\end{remark}

\subsection*{Main results for BMO drifts}

Although we are mainly interested in the dissipative SQG system, we emphasize that our method to prove \autoref{thm:PE} allows to consider more general drift-diffusion equations of the form \eqref{eq1}, as far as the regularity of the vector field $b$ is concerned. We recall that the SQG system relates via a Biot-Savart law the vector field $b$ to the solution $u$. Whenever such a law is not available, the problem becomes to ask under which minimal regularity  assumption on $b$ the equation \eqref{eq1} admits a local in time suitably regular solution. It is known since the seminal work by Caffarelli and Vasseur that dissipative SQG transports BMO-vector fields; which is more general than the standing assumption from our main results stated above that the drift is in $L^\infty$ in space and time. \\

The aim of this section is to state results under the aforementioned more general BMO assumption. While the adaptation of our techniques to $\text{BMO}$-drifts is rather straightforward when $s > 1/2$, the critical case $s = 1/2$ requires us to adjust the parabolic cylinders in space to the transport that is caused by the critical drift (see Lemma \ref{lemma:BMO-rescaling}). Similar arguments were already applied in \cite{CaffaVasseur}, \cite{sireDCDS}. In fact, let $z_r(t)$ be the solution to the ODE
\begin{align*}
z_r'(t) = \dashint_{B_1} b(r t , r x + r z_r(t)) \d x, ~~ z_r(0) = 0.
\end{align*}
Then given $t_0 \in \R$, $x_0 \in \R^d$ and $R > 0$, we introduce the slanted cylinders $\tilde{Q}_R(t_0,x_0) = I_R^{\ominus}(t_0) \times \tilde{B}_R(x_0)$ by
\begin{align*}
\tilde{Q}_R(t_0,x_0) = \{ (t, x + R z_R(t/R)) : (t,x) \in Q_R(t_0,x_0 )\}.
\end{align*}
Moreover, we define $\widetilde{\tail}(v(t);x_0,R) = \tail(v(t,\cdot + R z_R(t/R)); x_0,R)$, and introduce the following {\sl slanted} parabolic potential. Let $\mu \in \mathcal M(\mathbb R^{d+1})$ be a  Radon measure with finite total mass. Then we define
\begin{align*}
\widetilde{\mathcal P}^R_{1}[\mu](t_0,x_0):=\int_{0}^{R} \frac{|\mu|(\tilde{Q}_\rho(x,t))}{\rho^{d}} \frac{d\rho}{\rho}.
\end{align*}

Note that in view of \eqref{eq:zbounds}, the slanted potential can be controlled as follows
\begin{align} 
\label{eq:slpotest}
\begin{split}
	\widetilde{\mathcal {P}}^R_{1}[\mu](t_0,x_0) &\leq \int_{0}^{R}\frac{|\mu|(I_\rho^{\ominus}(t_0) \times B_{c\rho(C_1+C_2|\log(\rho)|)}(x_0))}{\rho^{d}}\frac{d\rho}{\rho} \\
	&\le \int_{0}^{f^{-1}(R)}\frac{|\mu|(I_{\rho}^{\ominus}(t_0) \times B_{\rho}(x_0))}{f^{-1}(\rho)^{d} f'(f^{-1}(\rho))}\frac{d\rho}{f^{-1}(\rho)},
	\end{split}
\end{align} 
where $f(\rho) = c \rho(C_1 + C_2 |\log(\rho)|)$. Here $c$ depends only on $d,\Lambda,q$, while $C_1$ and $C_2$ are given as in the statement of Theorem \ref{BMOdata}. We believe these estimate to be helpful when investigating mapping properties for the slanted potential, which would be required in order to obtain an analog of Corollary \ref{cor:Lorentzreg} for $\text{BMO}$-drifts in case $s = 1/2$.


We can then prove the following theorem:
 
 \begin{theorem}\label{BMOdata}
 Let $s \in [1/2,1)$ with $b \in L^\infty(\mathbb R;\mathrm{BMO}(\mathbb R^d))$. Assume furthermore that there exist constants $C_1,C_2 > 0$ such that 
 \begin{align}
\label{eq:BMO-ass}
\sup_{t,x_0} \dashint_{B_{1}(x_0)} |b(t,x)| \d x \le C_1, \qquad \Vert b \Vert_{L^{\infty}(\R;\mathrm{BMO}(\R^d))} \le C_2.
\end{align}
 Then the following holds:

 \begin{itemize}
 \item Suppose $s \in (1/2,1)$. Then the statements of Proposition \ref{prop:Holderq} and Theorem \ref{thm:PE} remain true under the more general assumptions on $b$ indicated above. The constants depend in addition on $C_1,C_2$.
 \item If $s = 1/2$ and $v$ is a weak solution to \eqref{eq:PDE-hom} in $Q_{\theta r}(t_0,x_0)$ for some $r > 0$, where $\theta = 1 + c_0(C_1 + C_2 |\log(r)|)$ for some $c_0 > 0$ depending only on $d$, then $v \in C^{\alpha}_{loc}(Q_{r}(t_0,x_0))$, and the following estimate holds true for any $q > 1$:
	\begin{align}
	\label{eq:Holderq-BMO}
	\begin{split}
		\Vert v \Vert_{C^{\alpha}(\tilde{Q}_{r/2}(t_0,x_0))} & \leq c r^{-\alpha} \Bigg[\left( \dashint_{\tilde{Q}_{r}(t_0,x_0)} |v| \d x \d t \right) \\&\quad + \left(\dashint_{I_r^{\ominus}(t_0)} \widetilde{\tail}(v(t);x_0,r)^q \d t \right)^{1/q} \Bigg]
		\end{split}
	\end{align}
	for some $c > 0$, $\alpha \in (0,1)$, depending only on $d,\Lambda,q,C_1,C_2$.\\
 Moreover, for any weak solution $u$ to \eqref{eq1} in $Q_{\theta r}(t_0,x_0)$ we have
 \begin{align}
	\label{eq:potential-est-BMO}
	\begin{split}
 |u(t_0,x_0)| & \leq c \left(\dashint_{\tilde{Q}_{r}(t_0,x_0)} |u|^q \d x \d t \right)^{1/q} \\
 &\quad + c \left(\dashint_{I_r^{\ominus}(t_0)} \widetilde{\tail}(u(t);x_0,R)^q \d t \right)^{1/q} \\ & \quad + c \widetilde{\mathcal {P}}^{r}_{1}[\mu](t_0,x_0) ,
 \end{split}
 \end{align}
	where $c$ depends only on $d,\Lambda,q, C_1,C_2$.
 \end{itemize}
 \end{theorem}

\subsection*{Outline}
 
This article is structured as follows. In Section \ref{sec:Holder}, we  prove the localized H\"older regularity estimate Theorem \ref{prop:Holderq} for bounded drifts, and for drifts in BMO \eqref{eq:Holderq-BMO}. Section \ref{sec:comparison} establishes a comparison estimate between solutions to \eqref{eq1} and the homogeneous problem \eqref{eq:PDE-hom}. The potential estimate for bounded drifts \autoref{thm:PE}, and for drifts in BMO \eqref{eq:potential-est-BMO} are shown in Section \ref{sec:potential}. Finally, Section \ref{sec:Aronson} establishes the heat kernel upper bound \autoref{thm:uhkeintro}.

\section{Local H\"older estimates for homogeneous problems}
\label{sec:Holder}

H\"older regularity estimates for global solutions to \eqref{eq:PDE-hom} have already been established in \cite{sireDCDS}. Our goal in this section is two-fold: First to prove a localized version of such H\"older bounds; second, to accomplish this goal under low integrability assumptions on the tail term, leading to precise estimates that are suitable to subsequently derive potential estimates. We accomplish this by proving the desired H\"older estimates directly, without relying on the harmonic extension that was used in the previous works \cite{CaffaVasseur,sireDCDS}.

Indeed, for linear parabolic equations without drift, i.e. $b \equiv 0$, the H\"older regularity of weak solutions to \eqref{eq:PDE-hom} has been established in various previously mentioned articles by a direct approach using either a nonlocal adaptation of the De Giorgi- or the Moser iteration. The proof consists of two main parts: a local boundedness estimate and a proof of oscillation decay. In order to run the respective iteration arguments, one needs to test the weak formulation of \eqref{eq:PDE-hom} with suitable test-functions and derive corresponding Caccioppoli-type estimates. In our case, the main work consists in proving the energy estimates. Starting from those estimates, the derivation of the local boundedness and the oscillation decay goes by an adaptation of the techniques recently developed in \cite{KaWe23}, \cite{Lia22} (see also \cite{ByKy23}).

\begin{remark} 
The nonlocal $\log(u)$-lemma seems to fail for \eqref{eq:PDE-hom}. Therefore, we need to follow a De Giorgi-type approach which does not make use of such estimate in order to deduce the oscillation decay. We follow the technique developed in \cite{Lia22} for the fractional $p$-Laplacian.
\end{remark}

\subsection{Energy estimates}

We have the following energy estimate for solutions to \eqref{eq:PDE-hom}. Note that they do not differ from the respective estimates for solutions to \eqref{eq:PDE-hom} in case $b \equiv 0$. For now, we always assume that $b \in L^\infty(\mathbb{R}^{d+1})$ with $\Vert b \Vert_{L^{\infty}(\R^{d+1})} \le \Lambda$ until specified otherwise.

Recall the notation
\begin{align*}
Q_r(t_0,x_0) = (t_0 - r^{2s} , t_0) \times B_{r}(x_0).
\end{align*}
Moreover, we write
\begin{align*}
I^{\ominus}_r(t_0) := (t_0 - r^{2s} , t_0), \qquad I^{\oplus}_r(t_0) := (t_0, t_0 + r^{2s}).
\end{align*}

In the following, we will drop $t_0$ from the notation and also write $B_r(x_0) = B_r$ and $Q_r(t_0,x_0) = Q_r$, whenever no confusion can arise.

Moreover, we write
\begin{align}
\label{eq:energy-def}
\mathcal{E}^{K_t}(u,v) = \int_{\R^d} \int_{\R^d}(u(x) - u(y))(v(x) - v(y)) K_t(x,y) \d y \d x
\end{align}
for the energy associated to $\mathcal{L}_t$, and denote
\begin{align*}
\mathcal{E}^s_{B_r}(u,v) = \int_{B_r} \int_{B_r} (u(x) - u(y))(v(x) - v(y))|x-y|^{-d-2s} \d y \d x \qquad \forall r > 0.
\end{align*}

\begin{lemma}
\label{lemma:Cacc}
Let $v$ be a weak subsolution (supersolution) to \eqref{eq:PDE-hom} in $Q_{2R}$. Then, there exists a constant $c > 0$ depending only on $d,s,\Lambda$ such that for any $l \in \R$ and any $0 < \rho_1 \le r \le r+\rho_1 \le R$ and $0 < \rho_2 \le r \le r+\rho_2 \le R$ such that $I_{r+\rho_1}^{\ominus} \times B_{r+\rho_2} \subset Q_{2R}$:
\begin{align}
\label{eq:Cacc1}
\begin{split}
\sup_{t \in I^{\ominus}_r} & \int_{B_r} w_{\pm}^2(x,t) \d x + \int_{I^{\ominus}_r} \cE^s_{B_r}(w_{\pm}(t),w_{\pm}(t)) \d t \\
&\le c [\rho_2^{-2s} \vee ((r+\rho_1)^{2s} - r^{2s})^{-1}] \int_{I^{\ominus}_{r+\rho_1}} \int_{B_{r+\rho_2}} w^2_{\pm}(t,x) \d x \d t \\
&\quad+ c \left(\frac{r+\rho_2}{\rho_2}\right)^d \rho_2^{-2s} \int_{I^{\ominus}_{r+\rho_1}} \left(\int_{B_{r+\rho_2}} w_{\pm}(t,x) \d x \right)\tail(w_{\pm}(t);0,r+\rho_2) \d t,
\end{split}
\end{align} 
and 
\begin{align}
\label{eq:Cacc2}
\begin{split}
\sup_{t \in I_{r + \rho_1}^{\ominus}} & \int_{B_r} w_{\pm}^2(x,t) \d x + \int_{I_{r + \rho_1}^{\ominus}} \cE^s_{B_r}(w_{\pm}(t),w_{\pm}(t)) \d t \\
& \quad + r^{-d-2s} \int_{I_{r+\rho_1}^{\ominus}} \int_{B_r} \int_{B_r} w_{\pm}(t,x) w_{\mp}(t,y) \d y \d x \d t\\
&\le \int_{B_{r+\rho_2}} w_{\pm}^2(t_0 - (r + \rho_1)^{2s},x) \d x + c \rho_2^{-2s} \int_{I_{r + \rho_1}^{\ominus}} \int_{B_{r+\rho_2}} w^2_{\pm}(t,x) \d x \d t \\
&\quad+ c \left(\frac{r+\rho_2}{\rho_2}\right)^d \rho_2^{-2s} \int_{I_{r+\rho_1}^{\ominus}} \left(\int_{B_{r+\rho_2}} w_{\pm}(t,x) \d x \right) \tail(w_{\pm}(t);0,r+\rho_2) \d t,
\end{split}
\end{align} 
where we denote $w = u-l$.
\end{lemma}

\begin{remark}
Note that in the following proof we will test the weak formulation of \eqref{eq:PDE-hom} with functions depending on $u$. While strictly speaking, a priori we only have $u \in L^2(I;H^{s}(\Omega))$, testing with the solution itself can be rigorously justified regardless. In fact, this can be achieved by perturbing the equation with an artificial diffusion term of the type $\varepsilon \Delta u$ and then passing to the limit as $\varepsilon \to 0$, see \cite[Appendix C]{CaffaVasseur} for more details.
\end{remark}

\begin{proof}
Let us explain how to prove the desired estimates for subsolutions. The proof of \eqref{eq:Cacc1} and \eqref{eq:Cacc2} is classical in case $b \equiv 0$ and goes by testing the weak formulation with the test function $\phi(t,x) = \eta^2(t)\tau^2(x)w_+(t,x)$ for some cut-off function $\tau \in C_c^{\infty}(B_{r+\frac{\rho_2}{2}})$ with $\tau \equiv 1$ in $B_{r}$, $0 \le \tau \le 1$, and $|\nabla \tau| \le 4 \rho_2^{-1}$, and $\eta \in C^{\infty}(\R)$, which yields for any $t \in I_{r+\rho_1}^{\ominus}$:
\begin{align}
\label{eq:weak-form-test}
\begin{split}
\int_{B_{r+\rho_2}} \eta^2(t) \tau^2 w_{+}^2(t) \d x &+ \int_{I^{\ominus}_{r+\rho_1}} \eta^2 \left( \int_{B_{r+\rho_2}}(b,\nabla [\tau^2 w_{+}]) u \d x \right) \d t \\
&\quad +  \int_{I^{\ominus}_{r+\rho_1}} \eta^2 \cE^{K_s}(u,\tau^2 w_{+}) \d t \\
&\le \eta^2(t_0 - (r+\rho_1)^{2s}) \int_{B_{r+\rho_2}}  \tau^2 w_{+}^2(t_0 - (r+\rho_1)^{2s}) \d x \\
&\quad + 2 \int_{I^{\ominus}_{r+\rho_1}} \eta |\eta'| \left(\int_{B_{r+\rho_2}} \tau^2 w_{+}^2 \d x  \right)\d t.
\end{split}
\end{align}
Moreover, we recall the following estimate (see for instance \cite[Lemma 3.1]{KaWe22b})
\begin{align}
\label{eq:Cacc-KaWe22}
\begin{split}
\int_{I^{\ominus}_{r+\rho_1}}\eta^2 & \cE^s_{B_{r+\rho_2}} (\tau w_+ , \tau w_+ )\d t - \int_{I^{\ominus}_{r+\rho_1}}\eta^2 \cE^s_{B_{r+\rho_2}}(w_-,\tau w_+) \d t \\
& \le c \int_{I^{\ominus}_{r+\rho_1}}\eta^2 \cE^{K_s}(u,\tau^2 w_+) \d t+ c \rho_2^{-2s} \int_{I^{\ominus}_{r+\rho_1}}\eta^2 \left(\int_{B_{r+\rho_2}} w_+^2 \d x \right) \d t\\
&\quad + c \left(\frac{r+\rho_2}{\rho_2}\right)^d \rho_2^{-2s} \int_{I^{\ominus}_{r+\rho_1}} \eta^2 \left(\int_{B_{r+\rho_2}} w_+ \d x \right)\tail(w_+;0,r+\rho_2) \d t,
\end{split}
\end{align}
where we also used assumption (iii) on $K_t$.
These two ingredients are sufficient to prove the energy estimates in case $b \equiv 0$. In order to treat the present case including a drift term, we  observe
\begin{align*}
(b,\nabla \phi) u = \eta^2 w_{\pm}^2(b,\nabla \tau^2) + \eta^2 l (b,\nabla(\tau^2)) + \frac{1}{2} \eta^2 \tau^2 (b , \nabla w_{\pm}^2),
\end{align*}
and that therefore, after integrating by parts and using that $b$ is divergence free and bounded:
\begin{align}
\label{eq:Cacc-drift}
\begin{split}
\left|\int_{I^{\ominus}_{r+\rho_1}} \int_{B_{r+\rho_2}} (b,\nabla \phi) u \d x \d t \right| &\le \frac{3}{2} \int_{I^{\ominus}_{r+\rho_1}} \eta^2 \int_{B_{r+\rho_2}} |(b,\nabla \tau^2)| w_{\pm}^2 \d x \d t \\
&\le c \rho_2^{-1} \int_{I^{\ominus}_{r+\rho_1}} \int_{B_{r+\rho_2}}  w_{\pm}^2 \d x \d t.
\end{split}
\end{align}

The proof of \eqref{eq:Cacc1} follows now by combining the estimates \eqref{eq:weak-form-test}, \eqref{eq:Cacc-KaWe22}, and \eqref{eq:Cacc-drift}, taking the supremum in $t \in I_{r+\rho_1}^{\ominus}$ upon choosing $\eta(t_0 - (r+\rho_1)^{2s}) = 0$, $\eta \equiv 1$ in $I_{r}^{\ominus}$, $|\eta'| \le 4 ((r+\rho_1)^{2s} - r^{2s})^{-1}$, and $0 \le \eta \le 1$.
To prove \eqref{eq:Cacc2} we proceed in the exact same way, however we work with $\eta \equiv 1$.
\end{proof}

\subsection{Local boundedness}

The following is the main result of this subsection. It states that weak solutions to \eqref{eq:PDE-hom} are locally bounded in their solution domain.

\begin{lemma}[Local boundedness]
\label{lemma:locbd}
Let $u$ be a weak solution to \eqref{eq:PDE-hom} in $Q_{2R}$. Then, there exists a constant $c > 0$, depending only on $d,s,\Lambda$, such that 
\begin{align*}
\sup_{Q_{R/2}} |u| \le c \dashint_{Q_R} |u(t,x)| \d x \d t + c \dashint_{I_R^{\ominus}} \tail(u(t);R) \d t.
\end{align*}
\end{lemma}

The main ingredient in the proof is the following lemma:

\begin{lemma}
\label{lemma:locbd-subsol}
Let $u$ be a weak subsolution to \eqref{eq:PDE-hom} in $Q_{2R}$. Then, there exists a constant $c > 0$, depending only on $d,s,\Lambda$, such that 
\begin{align*}
\sup_{Q_{R/2}} u_+ \le c \left(\dashint_{Q_R} u^2_+(t,x) \d x \d t \right)^{1/2} + c \dashint_{I_R^{\ominus}} \tail(u_+(t);R) \d t.
\end{align*}
\end{lemma}

\begin{remark}
In particular, the above lemma implies that weak solutions to \eqref{eq:PDE-hom} in $Q$ are locally bounded, i.e. $u \in L^{\infty}_{loc}(Q)$.
\end{remark}

\begin{proof}[Proof of Lemma \ref{lemma:locbd-subsol}]
Having at hand the energy estimate \eqref{eq:Cacc1}, which is exactly the same as in case $b \equiv 0$, the proof follows by the same arguments as in \cite[Theorem 1.8]{KaWe23}. Let us give a brief sketch: First of all, we observe that with the help of \eqref{eq:Cacc1}, we can derive the following estimates for weak subsolutions $v$ to \eqref{eq:PDE-hom}:
\begin{align}
\label{eq:loc-bd-sup}
\sup_{Q_{R/2}} v_+ &\le c \left(\dashint_{Q_R} v^2_+(t,x) \d x \d t \right)^{1/2} + c \sup_{t \in I_R^{\ominus}} \tail(v_+(t);R),\\
\label{eq:loc-bd-mixed}
\sup_{t \in I_{R/2}^{\ominus}} \left( \dashint_{B_{R/2}} v_+^2(t,x) \d x\right)^{1/2} &\le \left(\dashint_{Q_R} v^2_+(t,x) \d x \d t \right)^{1/2} + c \dashint_{I_{R}^{\ominus}} \tail(v_+(t);R) \d t.
\end{align}
The first estimate can be achieved by a De Giorgi iteration scheme, based on \eqref{eq:Cacc1} (see \cite[Theorem 3.6]{KaWe22b}), without relying anymore on the equation. The second estimate is a direct consequence of \eqref{eq:Cacc1} applied with $l = 0$ and an interpolation argument. It can be found in \cite[Lemma 3.1]{KaWe23}. Moreover, note that if $w$ is a subsolution to $\partial_t w + (b,\nabla w) + \mathcal{L}_t w \le f$, for some $f \in L^{1,\infty}_{t,x}$, then $v(t,x) = w(t,x) - \int_{t_0}^t \Vert f_+(s) \Vert_{L^{\infty}(B_R)} \d s$ is a subsolution to \eqref{eq:PDE-hom}. Therefore, we can deduce from \eqref{eq:loc-bd-sup} that:
\begin{align}
\label{eq:loc-bd-RHS}
\sup_{Q_{R/2}} w_+ \le c \left(\dashint_{Q_R} w^2_+(t,x) \d x \d t \right)^{1/2} + c \sup_{t \in I_R^{\ominus}} \tail(w_+(t);R) + c\Vert f_+ \Vert_{L_t^{1}L_x^\infty(Q_R)}.
\end{align}
To prove the desired result, we will now take $\eta \in C_c^{\infty}(B_{3R/2})$ with $\eta = 1$ in $B_R$ and observe that 
\begin{align*}
\partial_t(\eta u) + (b, \nabla(\eta u)) + \mathcal{L}_t(\eta u) \le -(b, \nabla((1-\eta) u)) - \mathcal{L}_t ((1-\eta) u) =: f ~~ \text{ in } Q_{2R}.
\end{align*} 
Moreover, for $t \in I^{\ominus}_{2R}$ and $x \in B_{R/2}$ it holds $(b, \nabla((1-\eta) u)) = 0$ and therefore
\begin{align*}
\Vert f_+ \Vert_{L^{1,\infty}_{t,x}(Q_{R/2})} = \Vert [\mathcal{L}_t((1-\eta) u)]_+ \Vert_{L^{1,\infty}_{t,x}(Q_{R/2})} \le \dashint_{I_{R/2}^{\ominus}}\tail(u_+(t);R) \d t.
\end{align*}
Thus, by application of \eqref{eq:loc-bd-RHS} to $\eta u$, and estimating $\sup_{t \in I_R^{\ominus}} \tail((\eta u)_+(t);R)$ with the help of \eqref{eq:loc-bd-mixed}, we obtain the desired result.
\end{proof}

\begin{proof}[Proof of Lemma \ref{lemma:locbd}]
An application of Lemma \ref{lemma:locbd-subsol} yields
\begin{align*}
\sup_{Q_{R/2}} |u| \le c \left( \dashint_{Q_R} |u(t,x)|^2 \d x \d t \right)^{1/2} + c \dashint_{I_R^{\ominus}} \tail(u(t);R) \d t.
\end{align*}
The exponent in the first term on the right hand side can be lowered by a standard covering and iteration argument (see \cite[Theorem 6.2]{KaWe22b}).
\end{proof}

\subsection{H\"older regularity estimates}

The main goal of this subsection is to prove Theorem \ref{prop:Holderq}. We establish the oscillation decay by using a nonlocal modification of the parabolic De Giorgi iteration, as it has been carried out for instance in \cite{Lia22} under the additional assumption that the tails are bounded in time (see also \cite{APT22}, \cite{ByKy23}). As it was the case for the local boundedness, this proof also goes as in case $b \equiv 0$, since the equation will only be used through the energy estimates \eqref{eq:Cacc1} and \eqref{eq:Cacc2} from now on. However, some of the arguments need to be modified in order to deal with tails that are $L^q$ in time for some $q \in (1,\infty)$. Therefore, we sketch some parts of the proof.

\begin{lemma}
\label{lemma:growth-lemma1}
Let $u$ be a weak supersolution to \eqref{eq:PDE-hom} in $Q_{2R}$ and assume that $u \ge 0$ in $Q_{2R}$. Then, for any $\delta \in (0,1)$ and $H > 0$ there exists $\nu \in (0,1)$, depending only on $d,s,\Lambda,q$, such that the following holds true: If 
\begin{align*}
\left| \left\{ u \le H \right\} \cap I^{\ominus}_{\delta R} \times B_R \right| \le \nu |I^{\ominus}_{\delta R} \times B_R|,
\end{align*}
and
\begin{align}
\label{eq:tail-small1}
\left(\dashint_{I_{2R}^{\ominus}} \tail( u_-(t) ;2R)^q \d t \right)^{1/q} \le \delta^{\frac{2s}{q}} H,
\end{align}
then it holds:
\begin{align*}
u \ge \frac{H}{2} ~~ \text{ in } I^{\ominus}_{\delta R/2} \times B_{R/2}.
\end{align*}
\end{lemma}

\begin{proof}
The proof goes by a standard De Giorgi iteration argument. We treat the $L^q$-tails by a standard interpolation argument (see \cite{DiB93}). We define 
\begin{align*}
k_i &= \frac{H}{2} + \frac{H}{2^i}, ~~ \hat{k}_i = \frac{k_i + k_{i+1}}{2}, ~~ w_i = (u-\hat{k}_{i})_+, ~~ R_i = \frac{R}{2} + \frac{R}{2^i}, ~~ \hat{R}_i = \frac{R_i + R_{i+1}}{2}\\
A_i &= \frac{\left| \left\{ u \le k_i \right\} \cap I^{\ominus}_{\delta R_i} \times B_{R_i} \right|}{|I^{\ominus}_{\delta R_i} \times B_{R_i}|}, \quad B_i = \left( \dashint_{I_{\delta R_i}^{\ominus}} \left( \frac{|\{ u(t) \le k_i \} \cap B_{R_i}|}{|B_{R_i}|} \right)^{\frac{q}{q-1}} \d t \right)^{\frac{q-1}{q(1+\kappa)}},
\end{align*}
where $\kappa := \frac{2s}{d}\frac{q-1}{q}$, and deduce from \eqref{eq:Cacc1}:
\begin{align}
\label{eq:Cacc1-apply1}
\begin{split}
\sup_{t \in I^{\ominus}_{\delta R_{i+1}}} & \int_{B_{R_{i+1}}} w_{i}^2(x,t) \d x + \int_{I^{\ominus}_{R_{i+1}}} \cE^s_{B_{R_{i+1}}}(w_{i}(t),w_{i}(t)) \d t \\
&\le c 2^{i2s} (\delta R)^{-2s} \int_{I^{\ominus}_{\delta \hat{R}_i}} \int_{B_{\hat{R}_i}} w^2_{i}(t,x) \d x \d t \\
&\quad+ c 2^{i(d+2s)} \hat{R}_i^{-2s} \int_{I^{\ominus}_{\delta \hat{R}_i}} \int_{B_{\hat{R}_i}} w_{i}(t,x) \tail(w_{i}(t);0,\hat{R}_i) \d t\\
&\le c (2^{i2s} + 2^{id} ) R^{-2s} H^2 A_i |Q_{R_i}|\\
&\quad + c 2^{i(d+2s)} R^{-2s} H  \int_{I^{\ominus}_{\delta \hat{R}_i}} \tail(w_{i}(t);0,2R) \left( \int_{B_{\hat{R}_i}}  \mathbbm{1}_{\{ u \le k_i \} }(t,x) \d x \right) \d t\\
&\le c 2^{i(d+2s)} R^{-2s} H^2 A_i |Q_{R_i}| \\
&\quad + c 2^{i(d+2s)} R^{-2s} \delta^{\frac{2s}{q}} \left( \dashint_{I_{2R}^{\ominus}} \tail(u_-(t);2R)^q \d t \right)^{1/q} B_i^{1+\kappa} |Q_{R_i}|\\
&\le c 2^{i(d+2s)} R^{-2s} H^2 |Q_{R_i}| \left( A_i + B_i^{1+\kappa} \right),
\end{split}
\end{align} 
where we applied H\"older's inequality and used \eqref{eq:tail-small1}. Moreover, recall the following estimate, which follows from the fractional Sobolev inequality and H\"older interpolation:
\begin{align}
\label{eq:interpol1}
\begin{split}
A_{i+1} &|I_{\delta R_{i+1}}^{\ominus} \times B_{R_{i+1}}| \le c 2^{2i} H^{-2} (A_i|I_{\delta R_i}^{\ominus} \times B_{R_i}|)^{\frac{2s}{d+2s}} \Bigg[ \sup_{t \in I^{\ominus}_{R_{i+1}}} \int_{B_{R_{i+1}}} w_{i}^2(x,t) \d x \\
& + \int_{I^{\ominus}_{R_{i+1}}} \cE^s_{B_{R_{i+1}}}(w_{i}(t),w_{i}(t)) \d t + 2^{i2s} R^{-2s} \int_{I^{\ominus}_{R_{i+1}}} \int_{B_{R_{i+1}}} w^2_{i}(t,x) \d x \d t \Bigg].
\end{split}
\end{align}
A combination of \eqref{eq:Cacc1-apply1} and \eqref{eq:interpol1} yields for some $\gamma_1 > 1$ and $c_1 > 0$:
\begin{align}
\label{eq:iteration1}
A_{i+1} \le \frac{c_1}{\delta^{2s}} 2^{\gamma_1 i} \delta^{2s \iota} \left(A_i^{1 + \frac{2s}{d+2s}} + A_i^{\frac{2s}{d+2s}} B_i^{1+\kappa} \right),
\end{align}
where $\iota := \frac{2s}{d+2s} \wedge 1 - \frac{q-1}{q(1+\kappa)}$.
Similar to the proof of \eqref{eq:interpol1}, one obtains
\begin{align}
\label{eq:interpol2}
\begin{split}
 B_{i+1} &\le c 2^{(d+2)i} \delta^{-\frac{2s(q-1)}{q(1+\kappa)}} R^{-d} H^{-2} \Bigg[ \sup_{t \in I^{\ominus}_{R_{i+1}}} \int_{B_{R_{i+1}}} w_{i}^2(x,t) \d x \\
& + \int_{I^{\ominus}_{R_{i+1}}} \cE^s_{B_{R_{i+1}}}(w_{i}(t),w_{i}(t)) \d t + 2^{i2s} R^{-2s} \int_{I^{\ominus}_{R_{i+1}}} \int_{B_{R_{i+1}}} w^2_{i}(t,x) \d x \d t \Bigg],
\end{split}
\end{align}
which yields the following estimate after combination with \eqref{eq:Cacc1-apply1} for some $\gamma_2 > 1$ and $c_2 > 0$:
\begin{align}
\label{eq:iteration2}
B_{i+1} \le \frac{c_2}{\delta^{2s}} 2^{\gamma_2 i} \delta^{2s \iota} \left(A_i + B_i^{1+\kappa} \right).
\end{align}
Finally, note that by assumption, we have $A_0 \le \nu < 1$, and therefore
\begin{align*}
A_0 + B_0^{1+\kappa} \le A_0 + A_0^{\frac{q-1}{q}} \le 2 A_0^{\frac{q-1}{q}} \le 2 \nu^{\frac{q-1}{q}}.
\end{align*}
Thus, the desired result follows by an application of a classical iteration lemma (see \cite[Chapter 1, Lemma 4.2]{DiB93}), upon choosing $\nu > 0$ small enough, depending only on $d,s,q,c_1,c_2,\gamma_1,\gamma_2$.
\end{proof}

\begin{lemma}
\label{lemma:growth-lemma3}
Let $u$ be a weak supersolution to \eqref{eq:PDE-hom} in $Q_{2R}$ and assume that $u \ge 0$ in $Q_{2R}$. Let $\alpha \in (0,1]$. Then, there exist $\eps, \delta \in (0,1)$, depending only on $d,s,\Lambda,q,\alpha$, such that for any $H > 0$ the following holds true: If
\begin{align} 
\label{eq:measure-ass}
\left| \left\{ u(t_0,\cdot) \ge H \right\} \cap B_R \right| \ge \alpha |B_R|
\end{align}
for some $t_0 \in I_{2R}^{\ominus}$, then either
\begin{align}
\label{eq:tail-small3}
\left(\dashint_{I_{2R}^{\ominus}} \tail(u_-(t);2R)^q \d t \right)^{1/q} > \delta^{\frac{2s}{q}} H,
\end{align}
or for all $t \in I_{\delta R}^{\oplus}(t_0)$ with $I_{\delta R}^{\oplus}(t_0) \subset I_{2R}^{\ominus}$: 
\begin{align*}
\left| \left\{ u(t,\cdot) \ge \eps H \right\} \cap B_R \right| \ge \frac{\alpha}{2} |B_R|.
\end{align*}
\end{lemma}

\begin{proof}
Let us assume that \eqref{eq:tail-small3} does not hold true. Let $\sigma \in (0,1)$. We apply \eqref{eq:Cacc2} to obtain for any $t \in I_{\delta R}^{\oplus}(t_0)$
\begin{align*}
\int_{B_{(1-\sigma)R}} & (u -H)^2_-(t,x) \d x \le \int_{B_R} (u - H)_-^2(t_0,x) \d x \\
&\quad + c (\sigma R)^{-2s} \int_{I_{\delta R}^{\oplus}(t_0)} \int_{B_R} (u - H)^2_- \d x \d \tau \\
&\quad + c \sigma^{-d-2s} R^{-2s} \int_{I_{\delta R}^{\oplus}(t_0)} \int_{B_R} (u - H)_- \tail((u - H)_-;R) \d x \d \tau\\
&\le H^2 |B_R| \left(  (1 - \alpha) + c \sigma^{-2s} \delta^{2s} + c \sigma^{-d-2s} \delta^{2s} \right),
\end{align*}
where we also used \eqref{eq:measure-ass}, and applied similar arguments to derive the second estimate, as in the previous proofs. From here, the proof is standard (see \cite[Lemma 3.3]{Lia22}), namely, for any $\eps \in (0,1)$, we have
\begin{align*}
|\{ u(t,\cdot) \le \eps H \} \cap B_R| &\le |\{ u(t,\cdot) \le \eps H \} \cap B_{(1-\sigma)R}| + |B_R \setminus B_{(1-\sigma)R}|\\
&\le ((1-\eps) H)^{-2} \int_{B_{(1-\sigma)}} (u - H)^2_-(t,x) \d x +  d \sigma |B_R|\\
&\le (1-\eps)^{-2} \left( (1 - \alpha) + c \sigma^{-2s} \delta^{2s} + c \sigma^{-d-2s} \delta^{2s} + d\sigma \right) |B_R|,
\end{align*}
which yields the desired result upon choosing $\sigma , \eps, \delta$ suitably, depending only on $d,s,c,\alpha$.
\end{proof}

\begin{lemma}
\label{lemma:growth-lemma4}
Let $u$ be a weak supersolution to \eqref{eq:PDE-hom} in $Q_{2R}$ and assume that $u \ge 0$ in $Q_{2R}$. Let $\delta, \sigma \in (0,1]$ and $H > 0$. Then, the following holds true: If
\begin{align} 
\label{eq:measure-ass2}
\left| \left\{ u(t,\cdot) \ge H \right\} \cap B_R \right| \ge \alpha |B_R| ~~ \forall t \in I_{\delta R}^{\ominus}
\end{align}
for some $\alpha \in (0,1)$, then either
\begin{align}
\label{eq:tail-small4}
\left(\dashint_{I_{2R}^{\ominus}} \tail(u_-(t);2R)^q \d t \right)^{1/q} > \delta^{\frac{2s}{q}} H,
\end{align}
or there exists a constant $c > 0$, depending only on $d,s, \Lambda,q$, such that
\begin{align*}
\left| \left\{ u \le \frac{\sigma H}{4} \right\} \cap I^{\ominus}_{\delta R} \times B_R \right| \le c \frac{\sigma}{\delta^{2s} \alpha} |I^{\ominus}_{\delta R} \times B_R|.
\end{align*}
\end{lemma}

\begin{proof}
Let us assume that \eqref{eq:tail-small4} does not hold true. We define $w = u - \frac{\sigma H}{2}$ and apply \eqref{eq:Cacc2}, which yields:
\begin{align*}
R^{-d-2s} & \int_{I_{\delta R}^{\ominus}} \int_{B_R} \int_{B_R} w_-(t,x) w_+(t,y) \d y \d x \d t \le \int_{B_{2R}} w_-^2(t_0 - R^{2s},x) \d x \\
&\quad + c R^{-2s} \int_{I_{\delta R}^{\ominus}} \int_{B_{2R}} w_-^2 \d x \d t + c R^{-2s} \int_{I_{\delta R}^{\ominus}} \left(\int_{B_{2R}} w_-(t,x) \d x \right) \tail(w_-;R) \d t\\
&\le c (\sigma H)^2 (\delta R)^{-2s} |I^{\ominus}_{\delta R} \times B_R| + c (\sigma H) \delta^{-2s} R^d \dashint_{I_{R}^{\ominus}}\tail(w_-;2R) \d t\\
&\le c (\sigma H)^2 (\delta R)^{-2s} |I^{\ominus}_{\delta R} \times B_R|
\end{align*}
Here, we used the same arguments as in the proof of the previous lemmas to obtain the second estimate. Moreover, we used H\"older's inequality to estimate the tail term. Finally, by \eqref{eq:measure-ass2}, we conclude
\begin{align*}
\left| \left\{ u \le \frac{\sigma H}{4} \right\} \cap I^{\ominus}_{\delta R} \times B_R \right| &\le c(\sigma H)^{-1} \int_{I_{\delta R}^{\ominus}} \int_{B_R} w_-(t,x) \d x \d t \\
&\le c(\sigma\alpha)^{-1}  H^{-2} R^{-d} \int_{I_{\delta R}^{\ominus}} \int_{B_R} \int_{B_R} w_-(t,x) w_+(t,y) \d y \d x \d t\\
&\le c \frac{\sigma}{\delta^{2s} \alpha} |I^{\ominus}_{\delta R} \times B_R|,
\end{align*}
as desired.
\end{proof}

Let us state the following growth lemma, which is an immediate consequence of Lemmas \ref{lemma:growth-lemma1}, \ref{lemma:growth-lemma3}, and \ref{lemma:growth-lemma4}:

\begin{corollary}[Growth lemma]
\label{cor:growth-lemma}
Let $u$ be a weak supersolution to \eqref{eq:PDE-hom} in $Q_{2R}$ and assume that $u \ge 0$ in $Q_{2R}$. Let $ \alpha \in (0,1]$ and $H > 0$. Then, the following holds true: If for some $t_0 \in I_{R}^{\ominus}$
\begin{align} 
\label{eq:measure-ass5}
\left| \left\{ u \ge H \right\} \cap B_R \right| \ge \alpha |B_R|,
\end{align}
then there exist $\delta, \theta \in (0,1)$, depending only on $d,s,\Lambda,q,\alpha$ such that if $I_{\delta R}^{\oplus}(t_0) \subset Q_{2R}$, and 
\begin{align}
\label{eq:tail-small5}
\left(\dashint_{I_{2R}^{\ominus}} \tail(u_-(t);2R)^q \d t \right)^{1/q} \le \delta^{\frac{2s}{q}} \theta H,
\end{align}
then
\begin{align*}
u \ge \frac{\theta H}{8} ~~ \text{ in } I_{\delta R/2}^{\ominus}(t_0 + (\delta R)^{2s}) \times B_{R/2}.
\end{align*}
\end{corollary}

\begin{proof}
Let $\xi,\alpha \in (0,1]$. We choose $\delta,\eps \in (0,1)$ as in Lemma \ref{lemma:growth-lemma3} and obtain for all $t \in I_{\delta R}^{\oplus}(t_0)$
\begin{align*}
\left| \left\{ u(t,\cdot) \ge \eps H \right\} \cap B_R \right| \ge \frac{\alpha}{2} |B_R|.
\end{align*}
Next, we can apply Lemma \ref{lemma:growth-lemma4} and obtain that for $\sigma = C^{-1} \nu \delta \alpha$ (where $\nu$ denotes the constant from Lemma \ref{lemma:growth-lemma1}) and $C > 0$ such that $c \frac{\sigma}{\delta^{2s} \alpha} \le \nu$ (where $c > 0$ denotes the constant from \ref{lemma:growth-lemma4})
\begin{align*}
\left| \left\{ u \le \frac{\sigma \eps H}{4} \right\} \cap  I_{\delta R}^{\oplus}(t_0) \times B_R \right| \le c \frac{\sigma}{\delta^{2s} \alpha} |I_{\delta R}^{\oplus}(t_0) \times B_R| \le \nu |I_{\delta R}^{\oplus}(t_0) \times B_R|,
\end{align*}
Finally, we apply Lemma \ref{lemma:growth-lemma1}, which yields
\begin{align*}
u \ge \frac{\sigma \eps H}{8} ~~ \text{ in } Q_{R/2} ~~ \text{ in } I_{\delta R/2}^{\ominus}(t_0 + (\delta R)^{2s}) \times B_{R/2},
\end{align*}
as desired.
\end{proof}

We are now ready to prove the H\"older regularity estimate:

\begin{proof}[Proof of Theorem \ref{prop:Holderq}]
Having Corollary \ref{cor:growth-lemma} and Lemma \ref{lemma:locbd} at our disposal, the proof is standard (see for instance \cite{KaWe23, Lia22}): The proof goes by constructing sequences $(M_j)$ and $(m_j)$ that are non-decreasing and non-increasing, respectively, and to find a small $\gamma \in (0,1)$ and a large $\nu > 1$ such that for every $j \in \mathbb{N}$:
\begin{align}
\label{eq:osc-decay}
m_j \le u \le M_j ~~ \text{ in } Q_{\nu^{-j} R}, ~~ \text{ and } M_j - m_j = L \nu^{-\gamma j},
\end{align}
where
\begin{align*}
L := C_0 \Vert u \Vert_{L^{\infty}(Q_R)} + \left( \dashint_{I_{2R}^{\ominus}} \tail(u(t);2R)^q \d t \right)^{\frac{1}{q}}
\end{align*}
for some constant $C_0$ to be determined later. Once this is achieved, we obtain
\begin{align*}
|u(t,x) - u(s,y)| \le R^{-\gamma} L (|x-y| + |t-s|^{1/(2s)})^{\gamma}, 
\end{align*}
from which we conclude the desired result upon recalling Lemma \ref{lemma:locbd}.

We choose $\nu$ such that $ Q_{\nu^{-1} R} \subset I_{\delta R/2}^{\ominus} \times B_{\delta R/2}$, where $\delta \in (0,1]$ denotes the constant in Corollary \ref{cor:growth-lemma} corresponding to $\alpha = 1/2$.
To prove \eqref{eq:osc-decay}, we observe that upon setting $M_j = \nu^{-\gamma j}L/2$ and $m_j = - \nu^{-\gamma j}L/2$, \eqref{eq:osc-decay} holds true for every $j \le j_0$ if we choose $C_0 \ge 2 \nu^{\gamma j_0}$, where $j_0$ can be chosen arbitrarily. To prove \eqref{eq:osc-decay} for $j > j_0$, we proceed by induction. Let us assume that \eqref{eq:osc-decay} holds true for some $j \ge j_0$, and assume without loss of generality that
\begin{align}
\label{eq:wlog-osc}
|\{ u(t_0 - ,\cdot) \ge m_j + (M_j - m_j)/2 \} \cap B_{\nu^{-j} R}| \ge \frac{1}{2}|B_{\nu^{-j}R}|.
\end{align}
We apply the growth lemma Corollary \ref{cor:growth-lemma} to $v = u - m_j$, and deduce that $u \ge M_j +  \frac{\theta L \nu^{-\gamma j}}{16}$ in $Q_{\nu^{-(j+1)R}}$. This yields \eqref{eq:osc-decay} for $j+1$ after defining $M_{j+1} = M_j$ and $m_{j+1} = m_j + \frac{\theta L \nu^{-\gamma j}}{16}$. Note that in case \eqref{eq:wlog-osc} fails, we proceed in the same way but upon defining $v = M_j - u$. Note that the verification of the assumptions of Corollary \ref{cor:growth-lemma} is standard and goes exactly as in \cite{KaWe23}, relying on appropriate choices of the parameters $\gamma,\nu,j_0$.
\end{proof}

\subsection{The case of BMO drifts}

Let us explain the necessary modifications in the aforementioned proofs which allow us to prove H\"older regularity estimates (see Proposition \ref{prop:Holderq}) for solutions to \eqref{eq1} with $b \in L^{\infty}(\R;\text{BMO}(\R^d))$.

We state an adapted version of the energy estimate (see Lemma \ref{lemma:Cacc}).

\begin{lemma}
Let $v$ be a weak subsolution (supersolution) to \eqref{eq:PDE-hom} in $Q_{2R}$. Then, in the same setting as in Lemma \ref{lemma:Cacc}, if $s > 1/2$, the estimate \eqref{eq:Cacc1} and \eqref{eq:Cacc2} remain true. If $s = 1/2$, then
\begin{align}
\label{eq:Cacc1-BMO}
\begin{split}
\sup_{t \in I^{\ominus}_r} & \int_{B_r} w_{\pm}^2(x,t) \d x + \int_{I^{\ominus}_r} \cE^{1/2}_{B_r}(w_{\pm}(t),w_{\pm}(t)) \d t \\
&\le c [\rho_2^{-1} \vee ((r+\rho_1) - r)^{-1} \vee |\log_2(r+\rho_2)|(r+\rho_2) \rho_2^{-2}] \int_{I^{\ominus}_{r+\rho_1}} \int_{B_{r+\rho_2}} \hspace{-0.3cm} w^2_{\pm}(t,x) \d x \d t \\
&\quad+ c \left(\frac{r+\rho_2}{\rho_2}\right)^d \rho_2^{-1} \int_{I^{\ominus}_{r+\rho_1}} \left(\int_{B_{r+\rho_2}} w_{\pm}(t,x)\d x \right) \tail(w_{\pm}(t);0,r+\rho_2) \d t,
\end{split}
\end{align} 
and 
\begin{align}
\label{eq:Cacc2-BMO}
\begin{split}
\sup_{t \in I_{r + \rho_1}^{\ominus}} & \int_{B_r} w_{\pm}^2(x,t) \d x + \int_{I_{r + \rho_1}^{\ominus}} \cE^{1/2}_{B_r}(w_{\pm}(t),w_{\pm}(t)) \d t \\
& \quad + r^{-d-1} \int_{I_{r+\rho_1}^{\ominus}} \int_{B_r} \int_{B_r} w_{\pm}(t,x) w_{\mp}(t,y) \d y \d x \d t\\
&\le \int_{B_{r+\rho_2}} w_{\pm}^2(t_0 - (r + \rho_1),x) \d x \\
&\quad+ c \left[\rho_2^{-1} \vee |\log_2(r+\rho_2)|(r+\rho_2) \rho_2^{-2} \right]\int_{I_{r + \rho_1}^{\ominus}} \int_{B_{r+\rho_2}} w^2_{\pm}(t,x) \d x \d t \\
&\quad+ c \left(\frac{r+\rho_2}{\rho_2}\right)^d \rho_2^{-1} \int_{I_{r+\rho_1}^{\ominus}} \left( \int_{B_{r+\rho_2}} w_{\pm}(t,x) \d x \right) \tail(w_{\pm}(t);0,r+\rho_2) \d t,
\end{split}
\end{align} 
where $c > 0$ also depends on $C_1,C_2$.
\end{lemma}

\begin{proof}
We only need to explain how to modify \eqref{eq:Cacc-drift}. Note that \eqref{eq:BMO-ass} implies
\begin{align}
\label{eq:BMO-help}
\sup_{t,x_0} \left(\dashint_{B_{R}(x_0)} |b(t,x)|^{\frac{d}{s}} \d x \right)^{\frac{s}{d}} \le -\log_2(R)C,
\end{align}
since we have the following computation, based on John-Nirenberg lemma:
\begin{align*}
\left(\dashint_{B_{R}(x_0)} |b(t,x)|^{\frac{d}{s}} \d x \right)^{\frac{s}{d}} &\le \left(\dashint_{B_{R}(x_0)} |b(t,x) - (b(t))_{B_{R}(x_0)}|^{\frac{d}{s}} \d x \right)^{\frac{s}{d}} + (b(t))_{B_{R}(x_0)} \\
&\le c(1-\log_2(R))C_1 + C_2.
\end{align*}
We compute using H\"older-- and Sobolev inequality, and $\frac{d+2s}{2d} = \frac{s}{d} + \frac{1}{2}$:
\begin{align*}
\frac{3}{2} &\int_{I^{\ominus}_{r+\rho_1}} \eta^2 \int_{B_{r+\rho_2}} |(b,\nabla \tau^2)| w_{\pm}^2 \d x \d t \\
&\le \eps \int_{I^{\ominus}_{r+\rho_1}} \eta^2 \left( \int_{B_{r+\rho_2}} (\tau w_{\pm})^{\frac{2d}{d-2s}} \d x  \right)^{\frac{d-2s}{d}} \d t \\
&\quad + c(\eps) \int_{I^{\ominus}_{r+\rho_1}} \eta^2 \left( \int_{B_{r+\rho_2}} [(b,\nabla \tau) w_{\pm}]^{\frac{2d}{d+2s}} \d x \right)^{\frac{d+2s}{d}} \d t \\
&\le c \eps \int_{I^{\ominus}_{r+\rho_1}} \eta^2 \cE^s(\tau w_{\pm}, \tau w_{\pm}) \d t \\
&\quad + c(\eps) \int_{I^{\ominus}_{r+\rho_1}} \eta^2 \left(\int_{B_{r+\rho_2}} b^{\frac{d}{s}} \d x \right)^{\frac{2s}{d}} \left(\int_{B_{r+\rho_2}} |\nabla \tau|^2 w_{\pm}^2 \d x \right) \d t.
\end{align*}
Now, the first term can be absorbed and the second term can be estimated from above.\\
Indeed, if $s > 1/2$, then using \eqref{eq:BMO-help}, we obtain:
\begin{align*}
c(\eps) &\int_{I^{\ominus}_{r+\rho_1}} \eta^2 \left(\int_{B_{r+\rho_2}} b^{\frac{d}{s}} \d x \right)^{\frac{2s}{d}} \left(\int_{B_{r+\rho_2}} |\nabla \tau|^2 w_{\pm}^2 \d x \right) \d t\\
 & \le c(\eps) \int_{I^{\ominus}_{r+\rho_1}} \eta^2 \left(\sup_{x_0 \in B_{r + \rho_2} \setminus B_r}\int_{B_{\rho_2}(x_0)} b^{\frac{d}{s}} \d x \right)^{\frac{2s}{d}} \left(\int_{B_{r+\rho_2}} |\nabla \tau|^2 w_{\pm}^2 \d x \right) \d t\\
&\le c(\eps) \int_{I^{\ominus}_{r+\rho_1}} (-\log_2(\rho_2)) \rho_2^{2s} \left(\int_{B_{r+\rho_2}} |\nabla \tau|^2 w_{\pm}^2 \d x \right) \d t \\
&\le - c(\eps,C)\rho_2^{-2s} \int_{I^{\ominus}_{r+\rho_1}} \int_{B_{r+\rho_2}}  w_{\pm}^2 \d x \d t,
\end{align*}
where we used that $-\log_2(\rho_2) \rho_2^{2s-2} \le c \rho_2^{-2s}$.
If $s = 1/2$, then we estimate the second term as follows:
\begin{align*}
c(\eps) &\int_{I^{\ominus}_{r+\rho_1}} \eta^2 \left(\int_{B_{r+\rho_2}} b^{\frac{d}{s}} \d x \right)^{\frac{2s}{d}} \left(\int_{B_{r+\rho_2}} |\nabla \tau|^2 w_{\pm}^2 \d x \right) \d t\\
&\le c(\eps,C) \int_{I^{\ominus}_{r+\rho_1}} (-\log_2(r+\rho_2)) (r+\rho_2) \left(\int_{B_{r+\rho_2}} |\nabla \tau|^2 w_{\pm}^2 \d x \right) \d t \\
&\le - c(\eps,C)\log_2(r+\rho_2)(r+\rho_2) \rho_2^{-2} \int_{I^{\ominus}_{r+\rho_1}} \int_{B_{r+\rho_2}}  w_{\pm}^2 \d x \d t.
\end{align*}
This observation concludes the proof.
\end{proof}

Clearly, since Lemma \ref{lemma:Cacc} remains true for $\text{BMO}$-drifts if $s > 1/2$, the proof of the H\"older estimate \eqref{eq:Holderq-BMO} in this setting follows in the exact same way as described before. Thus, in the following, we restrict ourselves to proving Proposition \eqref{eq:Holderq-BMO} in the critical case $s = 1/2$.

\subsubsection{Local boundedness in case $s = 1/2$ with slanted cylinders}

Note that the pre-factors in \eqref{eq:Cacc1-BMO} and \eqref{eq:Cacc2-BMO} are different from the case of bounded drifts. When following the proof of the local boundedness estimate using De Giorgi iteration, and tracking the pre-factor, one easily sees that the following estimate holds true on scale one:
\begin{align}
\label{eq:loc-bd-BMO}
\sup_{Q_{1/2}} |u| \le c \dashint_{Q_1} |u(t,x)| \d x \d t + c \dashint_{I_1^{\ominus}} \tail(u(t);1) \d t,
\end{align}
where $c > 0$ depends on $d,\Lambda,C_1,C_2$.

To get a local boundedness estimate on scale $R \in (0,1)$, we will make use of the following rescaling, which leaves the mean value of the drift invariant (see \cite[page 1677-1678]{sireDCDS}, \cite{CaffaVasseur}):

\begin{lemma}
\label{lemma:BMO-rescaling}
Let $s = 1/2$ and assume \eqref{eq:BMO-ass}. Let $r \in (0,1)$, and $z_r(t)$ be the solution to the ODE
\begin{align*}
z_r'(t) = \dashint_{B_1} b(r t , r x + r z_r(t)) \d x, ~~ z_r(0) = 0.
\end{align*}
Then, there exists a constant $c > 0$, depending only on $d$ such that
\begin{equation} \label{eq:zbounds}
\Vert z_r \Vert_{C^1(I_1)} \le c( C_1 + C_2 |\log(r)| ) =: \theta  -1.
\end{equation}
Moreover, let $u$ be a solution to \eqref{eq1} in $Q_{\theta r}$ for some $r \in (0,1)$. Then, $u_r(t,x) := u(rt,rx+rz_r(t))$ satisfies
\begin{align*}
\partial_t u_r + (b_r,\nabla u_r) + \mathcal{L}_r u_r = 0 ~~ \text{ in }  Q_1,
\end{align*}
where $\mathcal{L}_r$ is an operator of the form \eqref{eq:L} with kernel $K_r$ given by
\begin{align*}
K_r(x,y) := r^{d+1} K_t(rx - rz_r(t),ry - r z_r(t)),
\end{align*}
satisfying the conditions (i), (ii), (iii). Moreover, the drift $b_r$ is defined as $b_r(t,x) := b(rt,rx+rz_r(t)) - z_r'(t)$ and satisfies
\begin{align*}
\divo(b_r) = 0, \quad \dashint_{B_1} b_r(t,x) \d x = 0 , \quad \Vert b_r \Vert_{L^{\infty}(\R,\mathrm{BMO}(\R^d))} &= \Vert b(rt,x) \Vert_{L^{\infty}(\R,\mathrm{BMO}(\R^d))},\\
 \Vert b_r \Vert_{L^{\infty}(I_1 ; L^q(B_1))} & \le c C_2,
\end{align*}
for any $q \ge 0$, where $c> 0$ depends only on $d,q$.
\end{lemma}

Moreover, given $t_0 \in \R$, $x_0 \in \R^d$ and $R > 0$, we recall the definition of the slanted cylinders $\tilde{Q}_R(t_0,x_0) = I_R^{\ominus}(t_0) \times \tilde{B}_{t,R}(x_0)$, which are given as follows
\begin{align*}
\tilde{Q}_R(t_0,x_0) = \{ (t, x + R z_R(t/R)) : (t,x) \in Q_R(t_0,x_0 )\}.
\end{align*}

By combination of the previous lemma with \eqref{eq:loc-bd-BMO}, we get the following local boundedness estimate on small scales in slanted cylinders:

\begin{lemma}
Let $s = 1/2$ and assume \eqref{eq:BMO-ass}. Let $u$ be a weak solution to \eqref{eq:PDE-hom} in $Q_{\theta R}$, where $\theta - 1 = c(C_1 + C_2 |\log(R)|)$, as in Lemma \ref{lemma:BMO-rescaling}. Then, there exists a constant $c > 0$, depending only on $d,s,\Lambda$, such that 
\begin{align}
\label{eq:loc-bd-BMO-scaled}
\sup_{\tilde{Q}_R} u \le c \dashint_{\tilde{Q}_{2R}} |u(t,x)| \d x \d t + c\dashint_{I^{\ominus}_2} \widetilde{\tail}(u(t);2) \d t.
\end{align}
\end{lemma}

\begin{proof}
Defining $u_R$ as in \autoref{lemma:BMO-rescaling} and applying \eqref{eq:loc-bd-BMO} to $u_R$ yields
\begin{align*}
\sup_{(t,x) \in Q_R} & |u(t,x + R z_R(t/R))| = \sup_{(t,x) \in Q_1} u_R \\
&\le c \dashint_{Q_{2}} |u_R(t,x)| \d x \d t + c\dashint_{I^{\ominus}_2} \tail(u_R(t);2) \d t \\
&\le c \dashint_{Q_{2R}} |u(t,x + R z_R(t/R))| \d x \d t + c\dashint_{I^{\ominus}_R} \tail(u(t,\cdot + R z_R(t/R));R) \d t
\end{align*}
for some constant $c > 0$, depending only on $d,\Lambda,C_1,C_2$. This implies the desired result.
\end{proof}

\subsubsection{H\"older regularity in case $s = 1/2$ with slanted cylinders}

Moreover, note that the preliminary growth lemmas from the previous section (see Lemma \ref{lemma:growth-lemma1}, Lemma \ref{lemma:growth-lemma3}, and Lemma \ref{lemma:growth-lemma4}) remain true for $R = 1$, and as a consequence, also the following counterpart of Corollary \ref{cor:growth-lemma} holds true on scale one:

\begin{corollary}[Growth lemma]
\label{cor:growth-lemma-BMO}
Let $s = 1/2$ and assume \eqref{eq:BMO-ass}. Let $u$ be a weak supersolution to \eqref{eq:PDE-hom} in $Q_{2}$ and assume that $u \ge 0$ in $Q_{2}$. Let $ \alpha \in (0,1]$ and $H > 0$. Then, the following holds true: If for some $t_0 \in I_{1}^{\ominus}$
\begin{align*} 
\left| \left\{ u \ge H \right\} \cap B_1 \right| \ge \alpha |B_1|,
\end{align*}
then there exist $\delta, \theta \in (0,1)$, depending only on $d,\Lambda,\alpha,q,C_1,C_2$, such that if $I_{\delta}^{\oplus}(t_0) \subset Q_{2}$, and 
\begin{align*}
\left(\dashint_{I_{2}^{\ominus}} \tail(u_-(t);2)^q \d t \right)^{1/q} \le \delta^{\frac{2s}{q}} \theta H,
\end{align*}
then
\begin{align*}
u \ge \frac{\theta H}{8} ~~ \text{ in } I_{\delta/2}^{\ominus}(t_0 + \delta^{2s}) \times B_{1/2}.
\end{align*}
\end{corollary}

The proof of Corollary \ref{cor:growth-lemma-BMO} goes in the same way as before.

We are now in position to explain the proof of the H\"older regularity estimate \eqref{eq:Holderq-BMO} in slanted cylinders:

\begin{proof}[Proof of \eqref{eq:Holderq-BMO}]
The proof is split into two steps:\\
\textbf{Step 1:} First, we mimic the proof of \cite[Theorem 4.1]{sireDCDS}, using Corollary \ref{cor:growth-lemma-BMO} on scale one, instead of the growth lemma in \cite{sireDCDS} for the harmonic extension. Note that we also need to make use of the suitable rescaling of \eqref{eq1}, which was introduced in Lemma \ref{lemma:BMO-rescaling} and preserves the quantities in \eqref{eq:BMO-ass} in the same way as in the proof of \cite[Theorem 4.1]{sireDCDS}.
The arguments from this article transferred to our setup (replacing also the $L^{\infty}$ global bound by the $L^q$-tail, as it appears in Corollary \ref{cor:growth-lemma-BMO}) yield that there exists $\delta > 0$ such that for any $(t,x) \in Q_{1/2}$ it holds:
\begin{align*}
\Vert u \Vert_{C^{\alpha}(Q_{\frac{\delta}{C_1}(t,x)})} \le c \left[\Vert u \Vert_{L^{\infty}(Q_{3/4})} + \left(\dashint_{I_1^{\ominus}} \tail(u(t);1)^q \d t \right)^{1/q} \right],
\end{align*}
where $c > 0$ depends only on $d,\Lambda,q,C_2$. From here,

\textbf{Step 2:}
The next step is to prove the H\"older estimate on scale $R$. To do so, we first observe that a simple covering argument, together with the local boundedness estimate \eqref{eq:loc-bd-BMO}, implies
\begin{align}
\label{eq:BMO-Holder-scale1}
\begin{split}
\Vert u \Vert_{C^{\alpha}(Q_{1/2})} &\le c \left[\Vert u \Vert_{L^{\infty}(Q_{3/4})} + \left(\dashint_{I_1^{\ominus}} \tail(u(t);1)^q \d t \right)^{1/q} \right]\\
&\le c \left[\dashint_{Q_1} |u| \d x \d t + \left(\dashint_{I_1^{\ominus}} \tail(u(t);1)^q \d t \right)^{1/q} \right],
\end{split}
\end{align}
where $c > 0$ depends only on $d,\Lambda,q,C_1,C_2$. Next, we apply Lemma \autoref{lemma:BMO-rescaling} and obtain:
\begin{align*}
\Vert u \Vert_{C^{\alpha}(\tilde{Q}_{R/2})} &= \Vert u_R \Vert_{C^{\alpha}(Q_{1/2})} \\
&\le c \left[\dashint_{Q_1} |u_R| \d x \d t + \left(\dashint_{I_1^{\ominus}} \tail(u_R(t);1)^q \d t \right)^{1/q} \right] \\
&= c \left[\dashint_{\tilde{Q}_R} |u| \d x \d t + \left(\dashint_{I_R^{\ominus}} \widetilde{\tail}(u(t);1)^q \d t \right)^{1/q} \right],
\end{align*}
as desired.
\end{proof}

\section{Comparison estimates}
\label{sec:comparison}

The goal of this section is to prove the following comparison estimate, which compares a solution to \eqref{eq1} with a solution to the homogeneous equation \eqref{eq:PDE-hom}. Note that in this section we do not use any other property of $b$ except for $\operatorname{div}(b) = 0$. Thus, the same proof works for $L^{\infty}$-- and $\text{BMO}$-drifts. For a comment on how \eqref{eq:testing-2-SQG} might be additionally adapted in case $s = 1/2$ when $b \in \text{BMO}$, we refer to the end of this section (see \eqref{eq:testing-2-SQG-BMO}).

\begin{lemma}[Comparison estimate]
	\label{thm:comparison-SQG}
	Let $u$ be a weak solution to 
	\begin{align*}
		\partial_t u + (b,\nabla u) + \mathcal{L}_t u = \mu \in \mathcal{M}(\mathbb{R}^{d+1}) ~~ \text{ in } Q_{2r},
	\end{align*}
	and $v$ be the unique weak solution to
	\begin{equation} \label{eq:approx}
	\begin{aligned}
		\begin{cases}
			\partial_t v + (b,\nabla v) + \mathcal{L}_t v &= 0 ~~ \text{ in } Q_r,\\
			v &= u ~~ \text{ in } I_r^{\ominus} \times (\R^d \setminus B_r),\\
			v &= u ~~ \text{ in } \{-r^{2s}\} \times \R^d.
		\end{cases}
	\end{aligned}
	\end{equation}
	We set $w := u-v$. Then, we have
	\begin{align}
		\label{eq:testing-2-SQG}
		\sup_{t \in I_r^{\ominus}} \dashint_{B_r} |w| \leq c \left( \frac{|\mu|(Q_r)}{|B_r|} \right),
	\end{align}
	where $c > 0$ is a constant depending only on $d,s,\Lambda$.
\end{lemma}

\begin{remark}
	The existence of the weak solution $v$ to the problem \eqref{eq:approx} can for instance be established by standard variational methods, see for instance \cite[Theorem A.3]{BLS21} for a similar existence proof in the setting of the parabolic fractional $p$-Laplacian.
\end{remark}

\begin{proof}
	Let us test the weak formulations for $u$ and $v$ with $\phi_{\eps}^{\pm} = \pm \left(1 \wedge \frac{w_{\pm}}{\eps} \right)$ and subtract the two resulting identities. This yields:
	\begin{align*}
		\int_{B_r} (\partial_t w) \phi^{\pm}_{\eps} dx &+ \int_{B_r} (b(t,x),\nabla \phi_{\eps}^{\pm}(t,x)) w(t,x) \d x\\
		&+\int_{\R^d} \int_{\R^d} (w(t,x) - w(t,y))(\phi^{\pm}_{\eps}(t,x) - \phi^{\pm}_{\eps}(t,y)) K_t(x,y) \d y \d x\\
		&= \int_{B_r} \phi_{\eps} \d \mu \leq |\mu(t)|(B_r),
	\end{align*}
	where we used that $|\phi_{\eps}^{\pm}| \leq 1$.
	Note that since we want to estimate the left hand side from below, we can neglect the second term, since by definition
	\begin{align*}
		(w(t,x) - w(t,y))(\phi^{\pm}_{\eps}(t,x) - \phi^{\pm}_{\eps}(t,y)) \ge 0.
	\end{align*}
	Moreover, the drift term can be treated as follows, using $\dive(b) = 0$:
	\begin{align*}
		\int_{B_r} (b(t,x),\nabla \phi_{\eps}^{\pm}(t,x)) w(t,x) \d x =  \frac{1}{2}\int_{\R^d} (b(t,x),\nabla ((\phi_{\eps}^{\pm})^2(t,x)) ) \d x = 0.
	\end{align*}
	Thus, we have
	\begin{align}
		\label{eq:weak-form-consequence-1-SQG}
		\int_{B_r} (\partial_t w) \phi^{\pm}_{\eps} \d x \leq |\mu(t)|(B_r).
	\end{align}
	Moreover, let us compute
	\begin{align*}
		(\partial_t w) \phi^{\pm}_{\eps} = \partial_t \int_0^{w_{\pm}} \left( 1 \wedge \frac{\sigma}{\eps} \right) \d \sigma.
	\end{align*}
	Now, given $\delta > 0$ and $\tau \in I_r^{\ominus}$, let us define
	\begin{align*}
		\eta_{\delta}(t) = 
		\begin{cases}
			1, ~~ t \leq \tau,\\
			1 - \eps^{-1}(\tau - t), ~~ \tau \leq t \leq \tau + \eps,\\
			0, ~~ t > \tau + \eps.
		\end{cases}
	\end{align*}
	We observe that after integration by parts:
	\begin{align*}
		\int_{I_r^{\ominus}} \eta_{\delta} & \int_{B_r} (\partial_t w) \phi^{\pm}_{\eps} \d x \d t \\
		&=  \int_{I_r^{\ominus}} (-\partial_t \eta_{\delta}) \int_{B_r} \left(\int_0^{w_{\pm}} \left( 1 \wedge \frac{\sigma}{\eps} \right) \d  \sigma \right) \d x \d t\\
		& - \eta_{\delta}(-r^{2s}) \int_{B_r} \left(\int_0^{w_{\pm}(-r^{2s},x)} \left( 1 \wedge \frac{\sigma}{\eps} \right) \d \sigma \right) \d x\\
		&= \int_{I_r^{\ominus}} (-\partial_t \eta_{\delta}) \int_{B_r} \left(\int_0^{w_{\pm}} \left( 1 \wedge \frac{\sigma}{\eps} \right) \d \sigma \right) \d x \d t.
	\end{align*}
	In the last step we used that $w(-r^{2s}) \equiv 0$. Let us now multiply \eqref{eq:weak-form-consequence-1-SQG} on both sides with $\eta_{\delta}$ and integrate over $I_r^{\ominus}$. Then, by plugging in the previous observation, we obtain
	\begin{align*}
		\int_{I_r^{\ominus}} (-\partial_t \eta_{\delta}) \int_{B_r} \left(\int_0^{w_{\pm}} \left( 1 \wedge \frac{\sigma}{\eps} \right) \d \sigma \right) \d x \d t \leq \int_{I_r^{\ominus}}\eta_{\delta}(t) |\mu(t)|(B_r) \d t \leq |\mu|(Q_r),
	\end{align*}
	where we also used that $|\eta_{\delta}| \leq 1$.
	Now, taking the limit $\eps \searrow 0$, we observe that by the monotone convergence theorem:
	\begin{align*}
		\int_{I_r^{\ominus}} (-\partial_t \eta_{\delta}) \int_{B_r} w_{\pm} \d x \d t \leq |\mu|(Q_r).
	\end{align*}
	Moreover, since $-\partial_t \eta_{\delta} \to \delta_{\tau}$, as $\delta \searrow 0$, we obtain from the dominated convergence theorem
	\begin{align*}
		\int_{B_r} w_{\pm}(\tau,x) \d x \leq |\mu|(Q_r),
	\end{align*}
	and by adding the corresponding estimates for the positive and negative part, we obtain \eqref{eq:testing-2-SQG}, as desired.
\end{proof}

We close this section by commenting on the case $s = 1/2$ and $b \in \text{BMO}$. By carefully tracking the proof and replacing $Q_r$ by $\tilde{Q}_r$, it becomes apparent that the following result holds true instead of \eqref{eq:testing-2-SQG} when $u$ solves the equation in $\tilde{Q}_{2r}$ and $v$ is a solution in $\tilde{Q}_r$:
\begin{align}
\label{eq:testing-2-SQG-BMO}
\sup_{t \in I_r^{\ominus}} \dashint_{\tilde{B}_{t,r}} |w| \leq c \left( \frac{|\mu|(\tilde{Q}_r)}{|B_r|} \right).
\end{align}
Here the balls $\tilde{B}_{t,r}$ are defined as in \eqref{modballs} below.

\section{Proof of potential estimates}
\label{sec:potential}

We are now in position to prove the excess decay lemma given by Lemma \ref{lemma:OscDecIntro}, which turns out to be the crucial ingredient for the proof of the potential estimates given by Theorem \ref{thm:PE}. 


\begin{proof}[Proof of Lemma \ref{lemma:OscDecIntro}]
	Let $v$ be the comparison solution as defined in Section 3 with respect to $r=R$. Splitting into annuli  and using the parabolic H\"older estimate (see Theorem \ref{prop:Holderq}) yields
	\begin{align*}
		& E(v,t_0,x_0,2^{-m}R) \\ & \leq \left (\dashint_{I^{\ominus}_{2^{-m}R}(t_0)} \left [ \dashint_{B_{2^{-m}R}(x_0)} |v - (v)_{Q_{2^{-m}R}(t_0,x_0)} | \d x \right ]^q \d t \right )^{1/q} \\
		& \quad + (2^{-m}R)^{2s} \sum_{k=1}^{m} \left (  \dashint_{I^{\ominus}_{2^{-m}R}(t_0)} \left [ \int_{B_{2^{k-m} R}(x_0) \setminus B_{2^{k-m-1} R}(x_0)} \frac{|v - (v)_{Q_{2^{-m} R}(t_0,x_0)}|}{|x_0-y|^{d+2s}}\d y \right ]^q \d t \right )^{1/q}\\
		& \quad + (2^{-m}R)^{2s(1-1/q)} \left ( \int_{I^{\ominus}_{2^{-m}R}(t_0)} \left [ \int_{\R^d \setminus B_{R}(x_0)} \frac{|v - (v)_{Q_{2^{-m} R}(t_0,x_0)}|}{|x_0-y|^{d+2s}}\d y \right ]^q \d t \right )^{1/q} \\
		& \leq C \sum_{k=0}^{m-1} 2^{-2sk} \left (  \dashint_{I^{\ominus}_{2^{-m}R}(t_0)} \left [ \dashint_{B_{2^{k-m} R}(x_0)} |v - (v)_{Q_{2^{-m}R}(t_0,x_0)}|\d y \right ]^q \d t \right )^{1/q} \\
		& \quad + C 2^{-2s(1-1/q)m} \left (  \dashint_{I^{\ominus}_{R}(t_0)} \left [ \dashint_{B_{R}(x_0)} |v - (v)_{Q_{2^{-m}R}(t_0,x_0)}|\d y \right ]^q \d t \right )^{1/q}
		\\ & \quad+ 2^{-2s(1-1/q)m} ||\tail(v(\cdot)-(v)_{Q_{2^{-m}R}(t_0,x_0)};x_0,R)||_{L^q(I_R^{\ominus}(t_0))} \\
		& \leq C \sum_{k=0}^{m-1} 2^{-2sk} \osc_{Q_{2^{k-m}R}(t_0,x_0)} v + C 2^{-2s(1-1/q)m} E(v,t_0,x_0,R) \\
		& \leq C 2^{-\alpha m} \sum_{k=0}^{\infty} 2^{(\alpha-2s)k} E(v,t_0,x_0,R) + C 2^{-2s(1-1/q)m} E(v,t_0,x_0,R) \\
		& \leq C 2^{-\alpha m} E(v,t_0,x_0,R).
	\end{align*}
	
	Together with the comparison estimate from Lemma \ref{thm:comparison-SQG}, we obtain that
	\begin{align*}
		& E(u,t_0,x_0,2^{-m}R) \\
		& \leq E(v,t_0,x_0,2^{-m}R) + E(u-v,t_0,x_0,2^{-m} R) \\
		& \leq C 2^{-\alpha m} E(v,t_0,x_0,R) + C 2^{((d+2s)/q)m} E(u-v,t_0,x_0,R) \\
		& \leq C 2^{-\alpha m} E(u,t_0,x_0,R) + C 2^{((d+2s)/q)m} R^{-d} |\mu|(Q_{R}(t_0,x_0)),
	\end{align*}
	finishing the proof.
\end{proof}

\begin{proof}[Proof of Theorem \ref{thm:PE}]
	
	Let $m \geq 1$ to be chosen large enough in a way such that $m$ only depends on $d,s,\Lambda$. In particular, we require $m$ to be large enough such that
	\begin{equation} \label{llarge}
		C_0 2^{-\alpha m/2} \leq \frac{1}{2}.
	\end{equation}
	Observe that this in particular implies that for any integer $i \geq 0$,
	\begin{equation} \label{llargec}
		C_0 2^{-\alpha m(i+1)/2} \leq 2^{-(i+1)}.
	\end{equation}
	Moreover, for any integer $i \geq 0$, set $$
	E_i:=E(u,t_0,x_0,2^{-im} R).$$ By Lemma \ref{lemma:OscDecIntro} and \eqref{llarge}, \eqref{llargec}, for any integer $i \geq 0$ we have
	\begin{equation} \label{Ejdec}
		\begin{aligned}
			E_{i+1} & \leq \frac{1}{2} E_{i} + C_0 2^{((d+2s)/q)m} (2^{-im}R)^{-d} |\mu|(Q_{2^{-im} R}(t_0,x_0)) .
		\end{aligned}
	\end{equation}
	For $l\in \mathbb{N}$, summing \eqref{Ejdec} over $i \in \{0,...,l-1\}$ leads to
	\begin{align*}
		\sum_{i=1}^{l} E_{i} & \leq \frac{1}{2} \sum_{i=0}^{l-1} E_{i} + C_0 2^{((d+2s)/q)m} \sum_{i=0}^{l-1} (2^{-im}R)^{-d} |\mu|(Q_{2^{-im} R}(t_0,x_0)),
	\end{align*}
	so that by reabsorbing the first term on the right-hand side of the previous inequality we deduce
	\begin{align*}
		\sum_{i=1}^{l} E_{i} & \leq 2 E_0 + 2C_0 2^{((d+2s)/q)m} \sum_{i=0}^{\infty} (2^{-im}R)^{-d} |\mu|(Q_{2^{-im} R}(t_0,x_0)).
	\end{align*}
	
	Thus, we have
	\begin{align*}
		& |(u)_{Q_{2^{-(l+1)m}R}(t_0,x_0)}| \\ & \leq \sum_{i=0}^l (|(u)_{Q_{2^{-(i+1)m}R}(t_0,x_0)}-(u)_{Q_{2^{-im}R}(t_0,x_0)}|)+|(u)_{Q_{R}(t_0,x_0)}| \\
		& \leq C_1 \sum_{i=1}^l E_i + \dashint_{Q_{R}(t_0,x_0)} |u| \d x \d t \\
		& \leq C \Bigg ( \left ( \dashint_{I_R^{\ominus}(t_0)} \left [ \dashint_{B_{R}(t_0,x_0)} |u| \d x \right ]^q dt \right )^{1/q} \\ & \quad + R^{-2s/q} ||\tail(u(\cdot)-(u)_{Q_{R}(t_0,x_0)};x_0,R)||_{L^q(I_R^{\ominus}(t_0))} \\ & \quad + \sum_{i=0}^{\infty} (2^{-im}R)^{-d} |\mu|(Q_{2^{-im} R}(t_0,x_0)) \Bigg ),
	\end{align*}
	where $C$ depends only on $d,s,\Lambda,q$, which is in particular because $m$ only depends on the aforementioned quantities. Since the last term on the right-hand side can be controlled by the parabolic potential of order $2s$, the proof is finished by applying the Lebesgue differentiation theorem. 
	
\end{proof}

We end this section by commenting on how to obtain the parabolic potential estimates in case $b  \equiv 0$ for the full range $s \in (0,1)$ (see Remark \ref{remark:without-drift}).

\begin{remark}
\label{remark:proof-without-drift}
In case $b \equiv 0$, and $s \in (0,1)$, the H\"older estimate Proposition \ref{prop:Holderq} was obtained in \cite{KaWe23}. Moreover, it is easy to see that Lemma \ref{thm:comparison-SQG} can be proved in the exact same way in this case. Having at hand these two results, one can follow the proofs of Lemma \ref{lemma:OscDecIntro} and \autoref{thm:PE} line by line and thereby deduce the potential estimate for parabolic equations without drift and $s \in (0,1)$.
\end{remark}

\subsection{The case of BMO drifts}

Clearly, since Lemma \ref{lemma:Cacc} remains true for $\text{BMO}$ drifts if $s > 1/2$, the proof of the potential estimate \autoref{thm:PE} in this setting follows in the exact same way as described before. Thus, in the following, we restrict ourselves to proving \eqref{eq:potential-est-BMO} in the critical case $s = 1/2$.

%

\begin{lemma}[Excess decay under BMO drifts] \label{lemma:OscDecBMO} 
	Let $s=1/2$, assume that \eqref{eq:BMO-ass} holds, let $R>0$, $t_0 \in I$, $x_0 \in \mathbb{R}^n$, $q>1$ and $\mu \in \mathcal{M}(\mathbb{R}^{d+1})$. In addition, let $\alpha \in (0,s(1-1/q))$ be given by \eqref{eq:Holderq-BMO}. Let $u$ be a weak solution to \eqref{eq1} in $I \times \Omega$ and assume that $Q_{\theta R}(t_0,x_0) \subset I \times \Omega$, where $\theta - 1 = c(C_1 + C_2 |\log(R)|)$ for some $c > 0$, depending only on $d$.
	Then for any integer $m \geq 1$, we have
	\begin{equation} \label{eq:oscdecayBMO}
		\begin{aligned}
			\widetilde E(u,t_0,x_0,2^{-m} R) & \leq C_0 2^{-\alpha m} 	\widetilde E(u,t_0,x_0,R) \\ & \quad + C_0 2^{((d+2s)/q)m} R^{-d} |\mu|(\tilde Q_{R}(t_0,x_0)),
		\end{aligned}
	\end{equation}
	where $C_0 \geq 1$ depends only on $d,s,\Lambda,\alpha,q$ and for any $r>0$, the slanted excess is defined by
	\begin{align*}
		\widetilde E(u,t_0,x_0,r)& := \left (\dashint_{I_{r}^{\ominus}(t_0)} \left [ \dashint_{\tilde B_{t,r}(x_0)} |u - (u)_{\tilde Q_{r}(t_0,x_0)} | \d x \right ]^q \d t \right )^{1/q} \\ & \quad+ r^{-2s/q} ||\widetilde{\tail}(u(\cdot)-(u)_{\tilde Q_{r}(t_0,x_0)};x_0,r)||_{L^q(I_r^{\ominus}(t_0))},
	\end{align*}
	where for any $t \in \mathbb{R}$,
	\begin{equation} \label{modballs}
		\tilde{B}_{t,r}(x_0) = \{ x + r z_r(t/r) : x \in B_r(x_0 )\}.
	\end{equation}
\end{lemma}

\begin{proof}
	Let $v$ be the comparison solution as defined in Section 3 with respect to $r=R$. Splitting into annuli  and using the parabolic H\"older estimate in the case of BMO drifts (see \eqref{eq:Holderq-BMO}) yields
	\begin{align*}
		& \widetilde E(v,t_0,x_0,2^{-m}R) \\ & \leq \left (\dashint_{I^{\ominus}_{2^{-m}R}(t_0)} \left [ \dashint_{\tilde B_{t,2^{-m}R}(x_0)} |v - (v)_{\tilde Q_{2^{-m}R}(t_0,x_0)} | \d x \right ]^q \d t \right )^{1/q} \\
		& \quad + (2^{-m}R)^{2s} \\ & \quad \times \sum_{k=1}^{m} \left (  \dashint_{I^{\ominus}_{2^{-m}R}(t_0)} \left [ \int_{B_{2^{k-m} R}(x_0) \setminus B_{2^{k-m-1} R}(x_0)} \hspace{-2cm} \frac{|v(t,y+Rz_R(t/R)) - (v)_{\tilde Q_{2^{-m} R}(t_0,x_0)}|}{|x_0-y|^{d+2s}}\d y \right ]^q \d t \right )^{1/q}\\
		& \quad + (2^{-m}R)^{2s(1-1/q)} \left ( \int_{I^{\ominus}_{2^{-m}R}(t_0)} \left [ \int_{\R^d \setminus B_{R}(x_0)} \hspace{-1cm} \frac{|v(t,y+Rz_R(t/R)) - (v)_{\tilde Q_{2^{-m} R}(t_0,x_0)}|}{|x_0-y|^{d+2s}}\d y \right ]^q \d t \right )^{1/q} \\
		& \leq C \sum_{k=0}^{m-1} 2^{-2sk} \left (  \dashint_{I^{\ominus}_{2^{-m}R}(t_0)} \left [ \dashint_{\tilde B_{t,2^{k-m} R}(x_0)} |v - (v)_{\tilde Q_{2^{-m}R}(t_0,x_0)}|\d y \right ]^q \d t \right )^{1/q} \\
		& \quad + C 2^{-2s(1-1/q)m} \left (  \dashint_{I^{\ominus}_{R}(t_0)} \left [ \dashint_{\tilde B_{t,R}(x_0)} |v - (v)_{\tilde Q_{2^{-m}R}(t_0,x_0)}|\d y \right ]^q \d t \right )^{1/q}
		\\ & \quad+ 2^{-2s(1-1/q)m} ||\widetilde{\tail}(v(\cdot)-(v)_{\tilde Q_{2^{-m}R}(t_0,x_0)};x_0,R)||_{L^q(I_R^{\ominus}(t_0))} \\
		& \leq C \sum_{k=0}^{m-1} 2^{-2sk} \osc_{\tilde Q_{2^{k-m}R}(t_0,x_0)} v + C 2^{-2s(1-1/q)m} \widetilde E(v,t_0,x_0,R) \\
		& \leq C 2^{-\alpha m} \sum_{k=0}^{\infty} 2^{(\alpha-2s)k} \widetilde E(v,t_0,x_0,R) + C 2^{-2s(1-1/q)m} \widetilde E(v,t_0,x_0,R) \\
		& \leq C 2^{-\alpha m} \widetilde E(v,t_0,x_0,R).
	\end{align*}
	
	Together with the slanted comparison estimate from \eqref{eq:testing-2-SQG-BMO}, we obtain that
	\begin{align*}
		& \widetilde E(u,t_0,x_0,2^{-m}R) \\
		& \leq \widetilde E(v,t_0,x_0,2^{-m}R) + \widetilde E(u-v,t_0,x_0,2^{-m} R) \\
		& \leq C 2^{-\alpha m} \widetilde E(v,t_0,x_0,R) + C 2^{((d+2s)/q)m} \widetilde E(u-v,t_0,x_0,R) \\
		& \leq C 2^{-\alpha m} \widetilde E(u,t_0,x_0,R) + C 2^{((d+2s)/q)m} R^{-d} |\mu|(\tilde Q_{R}(t_0,x_0)),
	\end{align*}
	finishing the proof.
\end{proof}

We are now ready to prove \eqref{eq:potential-est-BMO}.

\begin{proof}[Proof of \eqref{eq:potential-est-BMO}]
The proof easily follows after replacing the standard excess $E(u,t_0,x_0,r)$ in the proof of Theorem \ref{thm:PE} by the slanted excess $\widetilde E(u,t_0,x_0,r)$.
\end{proof}

\section{Upper heat kernel estimates for drift-diffusion equations}
\label{sec:Aronson}

Let $\mathcal{L}_t$ and $b$ be as before with $\Vert b \Vert_{L^{\infty}(\R^{d+1})} \le \Lambda$, and assume that $K_t$ satisfies (i), (ii), (iii) for some $s \in [1/2,1)$ and $\Lambda > 0$.
Given $\eta \in [0,\infty)$ and $y \in \R^d$, we define $(t,x) \mapsto p(\eta,x;t,y)$ as the solution to the following problem:
\begin{align}
\label{eq:CP}
\begin{cases}
\partial_t u + (b,\nabla u)+\mathcal{L}_t u &=0~~\text{in}~~[\eta,\infty)\times\mathbb{R}^d,\\
u(\eta) &= \delta_y.
\end{cases}
\end{align}

We call $p$ the heat kernel of the operator $\partial_t + (b,\nabla \cdot) + \mathcal{L}_t$.

The goal of this section is to prove \autoref{thm:uhkeintro}, i.e. to show that for any $\eta \le t \le T$ it holds:
\begin{align}
\label{eq:uhke}
p(\eta,x;t,y) \le c \left((t-\eta)^{-\frac{d}{2s}} \wedge \frac{t-\eta}{|x-y|^{d+2s}} \right) \quad \forall x,y \in \R^d
\end{align}
for some constant $c > 0$, depending only on $d,s,\Lambda,T$.\\

Note that given $\eta \in [0,\infty)$ and $y \in \R^d$, the function $(t,x) \mapsto p(\eta,y;t,x)$ solves the following dual problem:
\begin{align}
\label{eq:CP-dual}
\begin{cases}
\partial_t u - (b,\nabla u)+\mathcal{L}_t u &=0~~\text{in}~~[\eta,\infty)\times\mathbb{R}^d,\\
u(0) &= \delta_y.
\end{cases}
\end{align}
This is a simple consequence of the fact that $b$ is divergence free. Let us also introduce the notation $\widehat{p}(\eta,x;t,y) = p(\eta,y;t,x)$ and define the corresponding semigroups $(P_t)$ and $(\widehat{P}_t)$ as follows:
\begin{align*}
P_{\eta,t} f(x) = \int_{\R^d} p(\eta,x;t,y) f(y) \d y, ~~ \widehat{P}_{\eta,t}f(x) = \int_{\R^d} \widehat{p}(\eta,x;t,y) f(y) \d y, ~~ f \in L^2(\R^d).
\end{align*} 

Let us first collect a few basic properties of the heat kernel and the associated semigroup:

\begin{lemma}
\label{lemma:aux-heatkernel}
The heat kernel $p(\eta,x;t,y)$ exists under the assumptions of  Theorem \ref{thm:uhkeintro}, and it holds:
\begin{align}
\label{eq:int}
\int_{\R^d} p(\eta,x;t,y) \d y &= \int_{\R^d} p(\eta,x;t,y) \d x = 1,\\
\label{eq:on-diag}
0 \le p(\eta,x;t,y) &\le c (t - \eta)^{-\frac{d}{2s}},\\
\label{eq:semigroup}
p(\eta,x;t,y) &= \int_{\R^d} p(\eta,x;\tau,z) p(\tau,z;t,y) \d z ~~ \forall \eta < \tau < t.
\end{align}
Moreover, for any $u_0 \in L^2(\R^d)$, $P_{\eta,t}f(x) = \int_{\R^d} u_0(y) p(\eta,x;t,y) \d y$ and $\widehat{P}_{\eta,t}u_0(x) = \int_{\R^d} u_0(y) p(\eta,x;t,y) \d x$ solve \eqref{eq:CP} with $P_{\eta,\eta}u_0 = u_0$, and \eqref{eq:CP-dual} with $\widehat{P}_{\eta,\eta}u_0 = u_0$, respectively.
\end{lemma}

\begin{proof}
The proof of all properties is standard. For more details, we refer to \cite{MaMi13b}. Let us just mention that the upper bound in \eqref{eq:on-diag} is a direct consequence of Lemma \ref{lemma:locbd} applied with $Q^{\ominus}_{R/2}(t,x)$, where $R = (t-\eta)^{1/(2s)}$, and \eqref{eq:int}, i.e.
\begin{align*}
p(\eta,x;t,y) \le c (t-\eta)^{-\frac{d}{2s}} \sup_{\tau \in (\eta,t)} \int_{\R^d} p(\eta,z;\tau,y) \d y \le c (t-\eta)^{-\frac{d}{2s}}.
\end{align*}
\end{proof}

Before we start with the actual proof of \autoref{thm:uhkeintro}, we need to introduce corresponding truncated objects. First, given $\rho > 0$, we define
\begin{align*}
\mathcal{L}^{\rho}_t u(x) = \text{p.v.} \int_{B_{\rho}(x)} (u(x) - u(y)) K_t(x,y) \d y.
\end{align*}
The heat kernel and semigroup corresponding to the problem
\begin{align*}
\partial_t u + (b,\nabla u)+\mathcal{L}_t^{\rho} u &=0~~\text{in}~~[\eta,\infty)\times\mathbb{R}^d,
\end{align*}
are denoted as $p^{\rho}, \widehat{p}^{\rho}, (P^{\rho}_t), (\widehat{P}^{\rho}_t)$, and are defined in the same way as their non-truncated counterparts, replacing $\mathcal{L}_t$ by $\mathcal{L}_t^{\rho}$ in all definitions. Moreover, Lemma \ref{lemma:aux-heatkernel} remains true with $p$ replaced by $p^{\rho}$ and $P_{\eta,t}$ replaced by $P^{\rho}_{\eta,t}$.\\
Moreover, we introduce the truncated ''carr\'e du champ`` operator
\begin{align*}
\Gamma^{\rho}_s(v,v)(x) = \int_{B_{\rho}(x)}(v(x) - v(y))^2|x-y|^{-d-2s} \d y.
\end{align*}

We need the following truncated version of the local boundedness-type estimate Lemma \ref{lemma:locbd}:

\begin{lemma}
\label{lemma:local-boundedness-truncated}
There exists a constant $C > 0$, depending only on $d,s,\Lambda$, such that for every $t_0 \in (0,\infty)$, $x_0 \in \R^d$, and $\rho, R > 0$ with $R \le \rho/2 \wedge t_0^{1/(2s)}$, and every solution $u$ to $\partial_t u +(b,\nabla u) + \mathcal{L}_t^{\rho} u = 0$ in $Q_R(t_0,x_0)$ it holds:
\begin{equation}
\sup_{Q_{R/2}(t_0,x_0)} u \le C \left( \frac{\rho}{R} \right)^{\frac{d}{2}} R^{-\frac{d}{2}} \sup_{t \in I_{R}^{\ominus}(t_0)} \left( \int_{B_{2\rho}(x_0)} u^2(t,x) \d x \right)^{1/2}.
\end{equation}
\end{lemma}

\begin{proof}
Recall that we have the same energy estimate \eqref{eq:Cacc1} as in  case $b \equiv 0$. The energy estimate remains true if $\mathcal{L}_t$ is replaced by $\mathcal{L}^{\rho}_t$ with $\tail(w_+;0,r+\rho_2)$ replaced by $\int_{B_{\rho} \setminus B_{r+\rho_2}} w_+(y) |y|^{-d-2s} \d y$. From here, the proof follows along the exact same arguments as the proof of \cite[Lemma 2.4]{KaWe23b}.
\end{proof}

\subsection{Aronson's weighted estimate}

The following lemma is an adaptation of \cite[Lemma 3.1]{KaWe23b}. It is a nonlocal adaptation of Aronson's weighted $L^2$-estimate, which is the main ingredient in the proof of the upper heat kernel estimates. Due to the presence of the divergence free drift term, we get the additional summand $|\nabla H|$ in the first condition on $H$, compared to \cite{KaWe23b}.

\begin{lemma}[Aronson's weighted estimate]
\label{lemma:Aronson}
Assume that $\rho > 0$, $0 \le \eta < T$. Let $u \in L^{\infty}((\eta,T) \times \R^d)$ be a solution to the $\rho$-truncated  problem $\partial_t u+(b,\nabla u)+ \mathcal{L}_t^{\rho} u =0$ in $(\eta,T) \times \R^d$ with $u(\eta) = u_0$. Then there exists a constant $C > 0$, depending only on $d,s,\Lambda$, such that for every bounded function $H : [\eta,t] \times \R^d \to [0,\infty)$ satisfying
\begin{itemize}
\item $C \left(|\nabla H| + \Gamma^{\rho}_s(H^{1/2},H^{1/2})\right) \le -\partial_t H$ in $(\eta,t) \times \R^d$,
\item $H^{1/2} \in L^2((\eta,t);H^{s}(\R^d))$,
\end{itemize}
the following estimate holds true:
\begin{equation}
\label{eq:Aronson}
\sup_{\tau \in (\eta,t)} \int_{\R^d} H(\tau,x) u^2(\tau,x) \d x \le \int_{\R^d} H(\eta,x) u_0^2(x) \d x.
\end{equation}
\end{lemma}

\begin{proof}
The proof is very similar to the one of \cite[Lemma 3.1]{KaWe23b}. As in \cite{KaWe23b}, let us take $\gamma_R \in C^{\infty}_c(\R^d)$ is such that $\gamma_R \equiv 1$ in $B_{R-1}(0)$, $\gamma_R \equiv 0$ in $\R^d \setminus B_{R}(0)$, $0 \le \gamma_R \le 1$, $\vert \nabla\gamma_R\vert \le 2$ for $R > 0$.\\
Note that
\begin{align*}
\int_{\eta}^t & \int_{\R^d} (b,\nabla u(\tau,x)) \gamma_R u H(\tau,x) \d x \d \tau = \frac{1}{2}\int_{\eta}^t \int_{\R^d} (b,\nabla(u^2)) \gamma_R H(\tau,x) \d x \d \tau\\
&= \frac{1}{2}\int_{\eta}^t \int_{\R^d} (b,\nabla \gamma_R) u^2 H(\tau,x) \d x \d \tau + \frac{1}{2}\int_{\eta}^t \int_{\R^d} (b,\nabla H) u^2 \gamma_R \d x \d \tau.
\end{align*}
We test the weak formulation with $\phi = \gamma_R^2 H u$, and obtain from the previous computation, as well as the same arguments as in \cite{KaWe23b}:
\begin{align*}
&\sup_{\tau \in (\eta,t)}\int_{\R^d} u^2(\tau,x) \gamma_R^2(x) H(\tau,x) \d x \le 
\int_{\R^d} u^2_0(x) \gamma_R^2(x) H(\eta,x) \d x\\
&\qquad\quad + c_2 \int_{\eta}^t \int_{\R^d} \nabla(u^2)\gamma_R H  \d x \d \tau\\ 
&\qquad\quad + \int_{\eta}^t \int_{\R^d} u^2 \gamma_R^2 \left(2c_2\Gamma^{\rho}_s(H^{1/2},H^{1/2}) + 2c_2|\nabla H| + \partial_t H \right) \d x \d \tau\\
&\qquad\quad + 2c_2 \int_{\eta}^{t} \int_{\R^d}\int_{B_{\rho}(x)} u^2(\tau,x)\gamma_R^2(z)(H^{1/2}(\tau,x) - H^{1/2}(\tau,z))^2 \vert x-y \vert^{-d-2s} \d z \d x \d \tau\\
&\qquad\quad + 2c_2 \int_{\eta}^{t}\int_{\R^d} u^2  H \Gamma^{\rho}_s(\gamma_R,\gamma_R) \d x \d \tau\\
&\qquad\quad + 2c_2 \int_{\eta}^{t}\int_{\R^d}\int_{B_{\rho}(x)} u^2(\tau,z) H(\tau,x) (\gamma_R(x)-\gamma_R(z))^2 \vert x-y \vert^{-d-2s} \d z \d x \d \tau\\
&\qquad\quad + 2 c_2 \int_{\eta}^{t}\int_{\R^d} u^2 H |\nabla \gamma_R|  \d x \d \tau\\
&\qquad\le \int_{\R^d} u^2_0(x) \gamma_R^2(x) H(\eta,x) \d x\\
&\qquad\quad + \int_{\eta}^{t} \int_{\R^d} u^2 \gamma_R^2 \left(4c_2\Gamma^{\rho}_s(H^{1/2},H^{1/2}) + 4c_2 |\nabla H| + \partial_t H \right) \d x \d \tau\\
&\qquad\quad + 2c_2 \Vert u\Vert_{\infty}^2\int_{\eta}^{t} \int_{\R^d \setminus B_{\frac{R-1}{2}(0)}} \Gamma^{\rho}_s(H^{1/2},H^{1/2}) \d x \d \tau\\
&\qquad\quad + 4c_2 \Vert u\Vert_{\infty}^2 \int_{\eta}^{t}\int_{\R^d} \left(\Gamma^{\rho}_s(\gamma_R,\gamma_R) + |\nabla \gamma_R | \right) H  \d x \d \tau.
\end{align*}
By the definition of $\gamma_R$, and since $H \in L^1((\eta,t) \times \R^d)$ by assumption, we have that the last term in the previous estimate converges to zero, as $R \to \infty$. The second term goes to zero by the same arguments as in \cite[Lemma 3.1]{KaWe23b}. This implies the desired result.
\end{proof}

\subsection{Gluing lemma}

The goal of this section is to prove the following relation between the truncated heat kernel $p^{\rho}$ and $p$:

\begin{lemma}[Gluing lemma]
\label{lemma:gluing}
Let $\rho > 0$ and $0 \le \eta \le t$. Then,
\begin{align*}
p(\eta,x;t,y) \le p^{\rho}(\eta,x;t,y) + c (t-\eta) \rho^{-d-2s} \quad \forall x,y \in \R^d.
\end{align*}
Moreover, it holds
\begin{align*}
p^{\rho}(\eta,x;t,y) \le e^{C (t-\eta) \rho^{-2s}}p(\eta,x;t,y) \quad \forall x,y \in \R^d.
\end{align*}
The constants $c,C > 0$ depend only on $d,s,\Lambda$.
\end{lemma}

A central ingredient in its proof is the following parabolic maximum principle, which remains true in the presence of a divergence free drift $b$. For $\Omega \subset \R^d$ open, we define $H^s_{\Omega}(\R^d) = \{u \in H^s(\R^d) : u \equiv 0 ~~ \text{ in } \R^d \setminus \Omega\}$.

\begin{lemma}
\label{lemma:pmp}
Let $\Omega \subset \R^d$ be open. Assume that $u$ solves
\begin{equation}
\begin{cases}
&\partial_t u + (b,\nabla u) + \mathcal{L}_t u \le 0,~  \text{in } (\eta,\infty) \times \Omega,\\
&u_+(t) \in H^{s}_{\Omega}(\R^d) ~~ \forall t \in (\eta,\infty),\\
&u_+(t) \to 0 ~ \text{in } L^2(\Omega), ~ \text{as } t \searrow \eta.
\end{cases}
\end{equation}
Then $u \le 0$ a.e. in $(\eta,\infty) \times \Omega$. The same result holds for subsolutions to $\partial_t u + (b,\nabla u) + \mathcal{L}_t^{\rho} u \le 0$ in $(\eta,\infty) \times \Omega$.
\end{lemma}

\begin{proof}
First, we observe that
\begin{align*}
\int_{\R^d} (b,\nabla u) u_+ \d x = \frac{1}{2} \int_{\R^d} (b , \nabla(u^2)) \d x = 0
\end{align*}
since $b$ is divergence free. Moreover, note that $\cE^{K_t}(u(t),u_+(t))\ge 0$. Thus, testing the weak formulation for $u$ with $\phi = u_+$, we obtain for any $ \eta < t_1 < t_2$:
\begin{align*}
\int_{\Omega} u_+^2(t_2,x) \d x - \int_{\Omega} u_+(t_1,x) \d x \le 0.
\end{align*}
The desired result follows by taking the limit $t_1 \searrow \eta$.
\end{proof}

We are now in the position to give the  

\begin{proof}[Proof of Lemma \ref{lemma:gluing}]
Let $\Omega_n \subset \R^d$, $n \in \mathbb{N}$, be an increasing sequence of domains, such that $\Omega_n \nearrow \R^d$. Let $f \in L^1(\R^d)$ with $\Vert f \Vert_{L^1(\R^d)} \le 1$. We define $(t,x) \mapsto P^{\Omega_n}_{\eta,t}f(x)$ to be the solution to $\partial_t u + (b , \nabla u) + \mathcal{L}_t u = 0$ with $u(\eta) = f$, and $(t,x) \mapsto P^{\Omega_n,\rho}_{\eta,t}f(x)$ to solve the corresponding problem with $\mathcal{L}_t$ replaced by $\mathcal{L}_t^{\rho}$.\\
The proof of the result goes as in \cite[Lemma 2.2]{KaWe23b} and is based on an application of the parabolic maximum principle Lemma \ref{lemma:pmp} to the function $(t,x) \mapsto P_{\eta,t}^{\Omega_n} f(x) - u(t,x)$, where
\begin{equation*}
u(t,x) =  P^{\Omega_n,\rho}_{\eta,t}f(x) + C (t-\eta) \rho^{-d-2s} \phi_n(x),
\end{equation*}
and $\phi_n \in C_c^{\infty}(\R^d)$ with $0 \le \phi_n \le 1$, $\phi_n \equiv 1$ in $\Omega_n$.\\
In order to apply Lemma \ref{lemma:pmp}, it remains to prove that $u$ solves $\partial_t u + (b,\nabla u) + \mathcal{L}_t u \le 0$ in $(\eta,\infty) \times \Omega_n$. To prove it, we take an arbitrary test function $\psi \in H^s_{\Omega_n}(\R^d)$. First, note 
\begin{align*}
\cE^{K_t}(P_{\eta,t}^{\rho,\Omega_n}f,\psi) - \cE^{K_t^{\rho}}(P_{\eta,t}^{\rho,\Omega_n}f,\psi)  &\ge - c \rho^{-d-2s} \Vert \psi \Vert_{L^1(\R^d)} \Vert P_{\eta,t}^{\rho,\Omega_n}f  \Vert_{L^1(\R^d)} \\
&\ge - c \rho^{-d-2s} \Vert \psi \Vert_{L^1(\R^d)},
\end{align*}
where $\cE^{K_t}$ was defined in \eqref{eq:energy-def}, we wrote $K_t^{\rho} (x,y) = K_t(x,y) \mathbbm{1}_{|x-y| \le \rho}(x,y)$, and we used:
\begin{align*}
\Vert P_{\eta,t}^{\rho,\Omega_n}f \Vert_{L^1(\R^d)} \le \Vert P_{\eta,t}^{\rho}f \Vert_{L^1(\R^d)} \le \Vert \widehat{P}_{\eta,t}^{\rho} 1 \Vert_{L^{\infty}(\R^d)} \Vert f \Vert_{L^1(\R^d)} \le  1.
\end{align*}
Note that in the first estimate, we applied Lemma \ref{lemma:pmp}, and in the last estimate, we used \eqref{eq:int}.
Moreover, we have
\begin{align*}
\int_{\R^d} \phi_n \psi \d x \ge \Vert \psi \Vert_{L^1(\R^d)}, ~~ \int_{\R^d}(b,\nabla \phi_n)\psi \d x = 0, ~~ \cE^{K_t}(\phi_n,\psi) \ge 0
\end{align*}
by the definitions of $\phi_n,\psi$.
Therefore,
\begin{align*}
\int_{\R^d}&\partial_t u(t,x) \psi(x) \d x + \int_{\R^d} (b,\nabla u) \psi \d x + \cE^{K_t}(u,\psi) \\
&= \left[\int_{\R^d} \partial_t P_{\eta,t}^{\rho,\Omega_n}f(x) \psi(x) \d x + \int_{\R^d} (b,\nabla P_{\eta,t}^{\rho,\Omega_n}f) \psi \d x  + \cE^{K_t}(P_{\eta,t}^{\rho,\Omega_n}f,\psi) \right]\\
&\quad+  C \rho^{-d-2s} \left[ \int_{\R^d}\phi_n \psi \d x + (t-\eta) \int_{\R^d} (b,\nabla \phi_n) \psi \d x+ (t-\eta)\cE^{K_t}(\phi_n,\psi) \right]\\
&\ge - c \rho^{-d-2s} \Vert \psi \Vert_{L^1(\R^d)} + C \rho^{-d-2s} \Vert \psi \Vert_{L^1(\R^d)}\\
&\ge \rho^{-d-2s} \Vert \psi \Vert_{L^1(\R^d)} \left(C - c\right).
\end{align*}
Thus, we can choose $C > c$ large enough such that, indeed $(t,x) \mapsto P_{\eta,t}^{\Omega_n} f(x) - u(t,x)$ is a subsolution.\\
Therefore, by the maximum principle Lemma \ref{lemma:pmp}, we obtain:
\begin{align}
\label{eq:MDapprox}
P^{\Omega_n}_{\eta,t} f(x) \le P^{\Omega_n,\rho}_{\eta,t} f(x) + C(t-\eta) \rho^{-d-2s} \Vert f \Vert_{L^1(\R^d)}.
\end{align}
Taking $n \to \infty$ in \eqref{eq:MDapprox} implies by the same arguments as in \cite[Lemma 5.3]{KaWe23b}
\begin{equation*}
P_{\eta,t} f(x) \le P^{\rho}_{\eta,t} f(x) + C(t-\eta) \rho^{-d-2s} \Vert f \Vert_{L^1(\R^d)}.
\end{equation*}
This yields the desired result by recalling the relation between $P_{\eta},t f$ and $p$ in Lemma \ref{lemma:aux-heatkernel}, and choosing $f = \mathbbm{1}_A$ for any measurable set $A \subset \R^d$.

The proof of the second claim follows by an application of the parabolic maximum principle to $u(t,x) = P_{\eta,t}^{\rho,\Omega}f - P_t^{\Omega_n} f(x) e^{(t-\eta) \rho^{-2s}}$.
\end{proof}

\subsection{Upper estimates for the truncated heat kernel}

The goal of this section is to deduce upper estimates for the heat kernel $p^{\rho}$ corresponding to the truncated operator $\mathcal{L}_t^{\rho}$:

\begin{proposition}[Truncated heat kernel estimate]
\label{thm:offdiagtrunc}
Let $T > 0$. There exist $c > 0$, $\nu > 1$, depending only on $d,s,\Lambda,T$, such that for every $\rho > 0$, $0 \le \eta \le t \le T$, and $x,y \in \R^d$ with  $t-\eta \le \frac{1}{4\nu}\rho^{2s}$:
\begin{equation}
\label{eq:offdiagtrunc}
p^{\rho}(\eta,x;t,y) \le c (t-\eta)^{-\frac{d}{2s}} 2^{\frac{\vert x-y \vert}{12\rho}} \left( \frac{\rho^{2s}}{\nu (t-\eta)} \right)^{-\frac{\vert x-y \vert}{12\rho} + \frac{1}{2} + \frac{d}{4s}}.
\end{equation}
\end{proposition}

Given $y \in \R^d$, $\rho > 0$, $0 \le \eta \le t \le T$, with $t-\eta \le \frac{1}{4\nu}\rho^{2s}$, where $\nu > 1$, we set
\begin{equation}
\label{eq:H}
\begin{split}
H(\tau,x) &:= \left(\frac{\rho^{2s}}{\nu[2(t-\eta) - (\tau-\eta)]}\right)^{-1} \wedge \left(\frac{\rho^{2s}}{\nu[2(t-\eta) - (\tau-\eta)]}\right)^{-\frac{\vert x-y \vert}{3\rho}}\\
&= e^{-\log\left( \frac{\rho^{2s}}{\nu[2(t-\eta) - (\tau-\eta)]} \right) \left( \frac{\vert x-y \vert}{3\rho} \vee 1\right)}.
\end{split}
\end{equation}

The following lemma shows that $H$ satisfies the assumption of Lemma \ref{lemma:Aronson} if $\nu > 1$ is chosen appropriately.

\begin{lemma}
\label{lemma:H}
For every $C,T > 0$, there exists $\nu = \nu(d,s,C,T) > 0$ such that for every $\rho > 0$, $y \in \R^d$ and $0 \le \eta \le t \le T$ with $t-\eta \le \frac{1}{4\nu}\rho^{2s}$, the function $H$ defined above satisfies:
\begin{align*}
C\left[\Gamma^{\rho}_s(H^{1/2},H^{1/2}) + |\nabla H|\right] \le -\partial_t H ~~ \text{ in } (\eta,t) \times \R^d, \qquad  H^{1/2} \in L^2((\eta,t);H^{s}(\R^d)).
\end{align*}
\end{lemma}

\begin{proof}
Note that $H$ is the same function as in \cite{KaWe23b}. Therefore, we refer to \cite[Lemma 3.3]{KaWe23b} for the computations involving $\partial_t H$ and $\Gamma^{\rho}_s(H^{1/2},H^{1/2})$.
In case $|x-y| \le 2 \rho$, we have
\begin{align*}
-\partial_t H(\tau,x) = \nu \rho^{-2s}, ~~ \Gamma^{\rho}_s(H^{1/2},H^{1/2})(\tau,x) = |\nabla H(\tau,x)| = 0.
\end{align*}
In case $2 \rho \le |x-y| \le 3 \rho$, we have
\begin{align*}
-\partial_t H(\tau,x) = \nu \rho^{-2s}, ~~ \Gamma^{\rho}_s(H^{1/2},H^{1/2})(\tau,x) \le c \rho^{-2s}, ~~ |\nabla H(\tau,x)| = 0.
\end{align*}
In case $|x-y| > 3 \rho$, it holds
\begin{align*}
- \partial_t H(\tau,x) &= \frac{\vert x-y \vert}{3\rho [2(t-\eta) - (\tau-\eta)]}e^{-\log\left( \frac{\rho^{2s}}{\nu[2(t-\eta) - (\tau-\eta)]} \right) \left( \frac{\vert x-y \vert}{3\rho}\right)},\\
\Gamma^{\rho}_s(H^{1/2},H^{1/2}) &\le c \frac{\vert x-y \vert}{3\rho \nu [2(t-\eta) - (\tau-\eta)]}e^{-\log\left( \frac{\rho^{2s}}{\nu[2(t-\eta) - (\tau-\eta)]} \right) \left( \frac{\vert x-y \vert}{3\rho}\right)},\\
|\nabla H(\tau,x)| &\le c \frac{1}{3\rho} e^{-\log\left( \frac{\rho^{2s}}{\nu[2(t-\eta) - (\tau-\eta)]} \right) \left( \frac{\vert x-y \vert}{3\rho}\right)} \\
&\le c \frac{\vert x-y \vert}{3\rho \nu^{\frac{1}{2s}} [2(t-\eta) - (\tau-\eta)]}e^{-\log\left( \frac{\rho^{2s}}{\nu[2(t-\eta) - (\tau-\eta)]} \right) \left( \frac{\vert x-y \vert}{3\rho}\right)}.
\end{align*}
In the last step, we used that since $s \in [1/2,1)$:
\begin{align*}
[2(t-\eta) - (\tau - \eta)] \le c(s,T) (t-\eta)^{\frac{1}{2s}} \le \frac{c(s,T)}{\nu^{\frac{1}{2s}}}\rho \le \frac{c(s,T)}{\nu^{\frac{1}{2s}}} |x-y|,
\end{align*}
where $c(s,T) > 0$ is a constant. This proves the desired result upon choosing $\nu > 1$ accordingly.
\end{proof}

Next, we apply Lemma \ref{lemma:Aronson} to $H$. This yields the following lemma:

\begin{lemma}
\label{lemma:auxtrunc}
Let $y \in \R^d$, $\sigma,\rho > 0$ and $0 \le \eta \le T$. Let $u_0 \in L^2(\R^d)$ be such that $u_0 \equiv 0$ in $B_{\sigma}(y)$. Assume that $u \in L^{\infty}( (\eta,T) \times \R^d)$ is a weak solution to $\partial_t u \pm (b,\nabla u) + \mathcal{L}_t^{\rho} u = 0$ in $(\eta,T) \times \R^d$ with $u(\eta) = u_0$. Then there exist $\nu > 1, C > 0$, depending only on $d,s,\Lambda,T$, such that for every $t \in (\eta,T)$ with $ t-\eta \le \frac{1}{4\nu}\rho^{2s}$:
\begin{equation*}
\vert u(t,y) \vert \le C (t-\eta)^{-\frac{d}{4s}}2^{\frac{\sigma}{6\rho}} \left(\frac{\rho^{2s}}{\nu (t-\eta)}\right)^{-\frac{\sigma}{6\rho} + \frac{1}{2} + \frac{d}{4s}} \Vert u_0\Vert_{L^2(\R^d)}.
\end{equation*}
\end{lemma}

\begin{proof}
We apply Lemma \ref{lemma:Aronson} with $H$ as in \eqref{eq:H}. This is possible due to Lemma \ref{lemma:H}. It follows that for every $y \in \R^d$, $0 \le \eta \le t \le T$ with $t-\eta \le \frac{1}{4\nu}\rho^{2s}$, applying also Lemma \ref{lemma:local-boundedness-truncated} with $R = (t-\eta)^{\frac{1}{2s}}$, $t_0 = t$, $x_0 = y$:
\begin{equation*}
\begin{split}
|u(t,y)| &\le c_1 (t-\eta)^{-\frac{d}{4s}} \left( \frac{\rho^{2s}}{t-\eta} \right)^{\frac{d}{4s}} \sup_{\tau \in (\eta,t)}\left(\int_{B_{2\rho}(y)} u^2(\tau,x)  \d x \right)^{1/2}\\
&\le c_2 (t-\eta)^{-\frac{d}{4s}} \left( \frac{\rho^{2s}}{t-\eta} \right)^{\frac{d}{4s}} \left( \frac{\sup_{x \in \R^d \setminus B_{\sigma}(y)} H(\eta,x)}{\inf_{\tau \in (\eta,t), x \in B_{2\rho}(y)} H(\tau,x)} \right)^{1/2} \Vert u_0 \Vert_{L^2(\R^d)}
\end{split}
\end{equation*}
for some constants $c_1, c_2 > 0$.\\
Note that there exist $c_3, c_4 > 0$ such that for $x \in \R^d \setminus B_{\sigma}(y)$ it holds 
\begin{equation*}
H(\eta,x) \le c_3 \left(\frac{\rho^{2s}}{2\nu (t-\eta)}\right)^{-\frac{\sigma}{3\rho}}
\end{equation*}
and for $(\tau,x) \in [\eta,t] \times B_{2\rho}(y)$ we have:
\begin{equation*}
H(\tau,x)\ge c_4 \left(\frac{\rho^{2s}}{\nu (t-\eta)}\right)^{-1}.
\end{equation*}
This follows directly from the definition of $H$ and $t-\eta \le \frac{1}{4\nu}\rho^{2s}$. 
Together, we obtain 
\begin{equation*}
\vert u(t,y)\vert \le c_5 (t-\eta)^{-\frac{d}{4s}} 2^{\frac{\sigma}{6\rho}} \left(\frac{\rho^{2s}}{\nu (t-\eta)}\right)^{-\frac{\sigma}{6\rho} + \frac{1}{2} + \frac{d}{4s}} \Vert u_0 \Vert_{L^2(\R^d)}
\end{equation*}
for some constant $c_5 > 0$, as desired.
\end{proof}

Now, we are in the position to deduce off-diagonal heat kernel estimates for the truncated heat kernel $p^{\rho}$:

\begin{proof}[Proof of Proposition \ref{thm:offdiagtrunc}]
First, observe the following on-diagonal bound for $p^{\rho}$, which follows from \eqref{eq:on-diag} combined with the second estimate in Lemma \ref{lemma:gluing}:
\begin{align}
\label{eq:on-diag-trunc}
p^{\rho}(t,x,y) \le e^{C(t-\eta)\rho^{-2s}} (t-\eta)^{-\frac{d}{2s}}.
\end{align}
The on-diagonal bound \eqref{eq:on-diag-trunc} and \eqref{eq:int} immediately imply for every $0 \le \eta \le t \le T$ with $t-\eta \le \frac{1}{4\nu}\rho^{2s}$ and $x,y \in \R^d$:
\begin{align}
\label{eq:weakl2esttrunc1-nonsym}
\left(\int_{\R^d} p^{\rho}(\eta,x;t,z)^2 \d z\right)^{1/2} \le c_1 (t-\eta)^{-\frac{d}{4s}},\\
\label{eq:weakl2esttrunc2-nonsym}
\left(\int_{\R^d} p^{\rho}(\eta,z;t,x)^2 \d z\right)^{1/2} \le c_1 (t-\eta)^{-\frac{d}{4s}}
\end{align}
for some $c_1 > 0$. On the other hand, from Lemma \ref{lemma:auxtrunc}, it follows for every $0 \le \eta \le t \le T$ with $t-\eta \le \frac{1}{4\nu}\rho^{2s}$ and $y \in \R^d$:
\begin{align}
\label{eq:l2esttrunc1-nonsym}
\left(\int_{\R^d \setminus B_{\sigma}(y)} p^{\rho}(\eta,y;t,z)^2 \d z\right)^{1/2} \le c_2 (t-\eta)^{-\frac{d}{4s}}2^{\frac{\sigma}{6\rho}} \left(\frac{\rho^{2s}}{\nu (t-\eta)}\right)^{-\frac{\sigma}{6\rho} + \frac{1}{2} + \frac{d}{4s}},\\
\label{eq:l2esttrunc2-nonsym}
\left(\int_{\R^d \setminus B_{\sigma}(y)} p^{\rho}(\eta,z;t,y)^2 \d z\right)^{1/2} \le c_2 (t-\eta)^{-\frac{d}{4s}}2^{\frac{\sigma}{6\rho}} \left(\frac{\rho^{2s}}{\nu (t-\eta)}\right)^{-\frac{\sigma}{6\rho} + \frac{1}{2} + \frac{d}{4s}}
\end{align}
for some $c_2 > 0$. To prove \eqref{eq:l2esttrunc1-nonsym}, one observes that
\begin{align*}
u(\tau,x) = \int_{\R^d \setminus B_{\sigma}(y)} p^{\rho}(\eta,y;t,z)p^{\rho}(\eta,x;\tau,z)\d z
\end{align*}
solves $\partial_t u +(b,\nabla u) + \mathcal{L}_t^{\rho} u = 0$ and satisfies the assumptions of Lemma \ref{lemma:auxtrunc} with 
\begin{align*}
u_0(x) = p^{\rho}(\eta,y;t,x)\mathbbm{1}_{\{\vert x-y \vert > \sigma\}}(x), \qquad u(t,y) = \int_{\R^d \setminus B_{\sigma}(y)} p^{\rho}(\eta,y;t,z)^2 \d z.
\end{align*}
From here, \eqref{eq:l2esttrunc2-nonsym} follows by the same arguments, choosing 
\begin{align*}
\widehat{u}(\tau,x) = \int_{\R^d \setminus B_{\sigma}(y)} p^{\rho}(\eta,z;t,y)p^{\rho}(\eta,z;\tau,x)\d z,
\end{align*}
which solves the dual equation $\partial_t u -(b,\nabla u) + \mathcal{L}_t^{\rho} u = 0$ with $\widehat{u}_0(x) = p^{\rho}(\eta,x;t,y)\mathbbm{1}_{\{\vert x-y \vert > \sigma\}}(x)$.\\
Note that boundedness of $u$ and $\widehat{u}$, which is required in order to apply Lemma \ref{lemma:auxtrunc}, is a direct consequence of the on-diagonal bound \eqref{eq:on-diag-trunc} and integrability of the truncated heat kernel, which follows from \eqref{eq:int} and the second property in Lemma \ref{lemma:gluing}.\\
To prove \eqref{eq:offdiagtrunc}, let us fix $0 \le \eta \le t \le T$ with $t-\eta \le \frac{1}{4\nu}\rho^{2s}$, and $x,y \in \R^d$. Then we define $\sigma = \frac{1}{2}\vert x-y \vert$ and compute, using the semigroup property:
\begin{align*}
p^{\rho}(\eta,x;t,y) &= \int_{\R^d} p^{\rho}(\eta,x;(t-\eta)/2,z)p^{\rho}((t-\eta)/2,z;t,y) \d z\\
&= \int_{\R^d \setminus B_{\sigma}(y)}p^{\rho}(\eta,x;(t-\eta)/2,z)p^{\rho}((t-\eta)/2,z;ty) \d z \\
&\quad + \int_{B_{\sigma}(y)}p^{\rho}(\eta,x;(t-\eta)/2,z)p^{\rho}((t-\eta)/2,z;ty) \d z\\
& =: J_1+J_2.
\end{align*}
For $J_1$, we compute, using \eqref{eq:weakl2esttrunc1-nonsym}, \eqref{eq:l2esttrunc2-nonsym}
\begin{align*}
J_1 &\le \left(\int_{\R^d \setminus B_{\sigma}(y)} p^{\rho}(\eta,x;(t-\eta)/2,z)^2 \d z \right)^{1/2} \left(\int_{\R^d \setminus B_{\sigma}(y)}p^{\rho}((t-\eta)/2,z;t,y)^2 \d z \right)^{1/2}\\
&\le c_3 (t-\eta)^{-\frac{d}{2s}} 2^{\frac{\vert x-y \vert}{12\rho}} \left( \frac{\rho^{2s}}{\nu (t-\eta)} \right)^{-\frac{\vert x-y \vert}{12\rho} + \frac{1}{2} + \frac{d}{4s}}
\end{align*}
for some $c_3 > 0$. For $J_2$, observe that $B_{\sigma}(y) \subset \R^d \setminus B_{\sigma}(x)$, and therefore by \eqref{eq:weakl2esttrunc2-nonsym}, \eqref{eq:l2esttrunc1-nonsym}:
\begin{align*}
J_2 &\le \left(\int_{\R^d \setminus B_{\sigma}(x)} p^{\rho}(\eta,x;(t-\eta)/2,z)^2 \d z \right)^{1/2} \left(\int_{\R^d \setminus B_{\sigma}(x)} p^{\rho}((t-\eta)/2,z;t,y)^2 \d z \right)^{1/2}\\
&\le c_3 (t-\eta)^{-\frac{d}{2s}} 2^{\frac{\vert x-y \vert}{12\rho}} \left( \frac{\rho^{2s}}{\nu (t-\eta)} \right)^{-\frac{\vert x-y \vert}{12\rho} + \frac{1}{2} + \frac{d}{4s}}.
\end{align*}
Together, we obtain the desired result.
\end{proof}

\subsection{Proof of the upper heat kernel estimate}

Now, we are in the position to prove Theorem \ref{thm:uhkeintro}:

\begin{proof}[Proof of Theorem \ref{thm:uhkeintro}]
Let $x,y \in \R^d$ be fixed. By the on-diagonal estimate \eqref{eq:on-diag} it suffices to show that for some constants $c_0, c_1,c_2 > 0$ and $t-\eta \le c_0 \vert x-y \vert^{2s}$ it holds
\begin{equation*}
p(\eta,x;t,y) \le c_1\frac{t-\eta}{\vert x-y \vert^{d+2s}}.
\end{equation*}
By Lemma \ref{lemma:gluing}, we know that for every $\rho > 0$ and $0 \le \eta \le t \le T$ with $t-\eta \le \frac{1}{4\nu} \rho^{2s}$:
\begin{equation}
\label{eq:gluinglemma-nonsym}
p(\eta,x;t,y) \le p^{\rho}(\eta,x;t,y) + C (t-\eta) \rho^{-d-2s}
\end{equation}
for some $C > 0$.
We choose $\rho = \frac{\vert x-y \vert}{12}\left(\frac{d+2s}{2s} + \frac{1}{2} + \frac{d}{4s}\right)^{-1}$ and define $c_2 = 12 \left(\frac{d+2s}{2s} + \frac{1}{2} + \frac{d}{4s}\right)$. Then by combination of \eqref{eq:offdiagtrunc} and \eqref{eq:gluinglemma-nonsym}, we deduce for $t-\eta \le \frac{1}{4\nu}\rho^{2s}$:
\begin{equation*}
p(\eta,x;t,y) \le c_1 (t-\eta) \vert x-y \vert^{-d-2s},
\end{equation*}
where $c_1 >  0$. This proves the desired result.
\end{proof}

\section*{Acknowledgments}

S. N. gratefully acknowledges the support of the German Research Foundation - SFB 1283/2 2021 - 317210226. Y. S. is partially supported by NSF DMS grant $2154219$, " Regularity {\sl vs} singularity formation in elliptic and parabolic equations". Part of this work was done while S. N. was visiting Johns Hopkins University whose hospitality is acknowledged.
M. W. was supported by the European Research Council (ERC) under the Grant Agreement No 801867 and by the AEI project PID2021-125021NA-I00 (Spain). Q. H. N. is supported by the Academy of Mathematics and Systems Science, Chinese Academy of Sciences startup fund; CAS Project for Young Scientists in Basic Research, Grant No. YSBR-031;  and the National Natural Science Foundation of China (No. 12288201);  and the National Key R$\&$D Program of China under grant 2021YFA1000800.

\bibliographystyle{alpha} 
\bibliography{biblio}

\newcommand{\etalchar}[1]{$^{#1}$}
\begin{thebibliography}{BDGO97}

\bibitem[APT22]{APT22}
Karthik Adimurthi, Harsh Prasad, and Vivek Tewary.
\newblock H{\"o}lder regularity for fractional p-{L}aplace equations.
\newblock {\em arXiv:2203.13082}, 2022.

\bibitem[Aro68]{Aro68}
Donald Aronson.
\newblock Non-negative solutions of linear parabolic equations.
\newblock {\em Ann. Sc. Norm. Super. Pisa Cl. Sci.}, 22(4):607--694, 1968.

\bibitem[BCD{\etalchar{+}}18]{BCDKS}
Dominic Breit, Andrea Cianchi, Lars Diening, Tuomo Kuusi, and Sebastian
  Schwarzacher.
\newblock Pointwise {C}alder\'{o}n-{Z}ygmund gradient estimates for the
  {$p$}-{L}aplace system.
\newblock {\em J. Math. Pures Appl. (9)}, 114:146--190, 2018.

\bibitem[BDGO97]{BDG}
Lucio Boccardo, Andrea Dall'Aglio, Thierry Gallou\"{e}t, and Luigi Orsina.
\newblock Nonlinear parabolic equations with measure data.
\newblock {\em J. Funct. Anal.}, 147(1):237--258, 1997.

\bibitem[BK23]{ByKy23}
Sun-Sig Byun and Kyeongbae Kim.
\newblock H{\"o}lder estimate with an optimal tail for nonlocal parabolic
  p-{L}aplace equations.
\newblock 2023.
\newblock To appear in Ann. Mat. Pura Appl.

\bibitem[BLS21]{BLS21}
Lorenzo Brasco, Erik Lindgren, and Martin Str\"{o}mqvist.
\newblock Continuity of solutions to a nonlinear fractional diffusion equation.
\newblock {\em J. Evol. Equ.}, 21(4):4319--4381, 2021.

\bibitem[BY19]{BYP}
Sun-Sig Byun and Yeonghun Youn.
\newblock Potential estimates for elliptic systems with subquadratic growth.
\newblock {\em J. Math. Pures Appl. (9)}, 131:193--224, 2019.

\bibitem[CCV11]{CCV}
Luis Caffarelli, Chi~Hin Chan, and Alexis Vasseur.
\newblock Regularity theory for parabolic nonlinear integral operators.
\newblock {\em J. Amer. Math. Soc.}, 24(3):849--869, 2011.

\bibitem[Cia11]{Cianchi}
Andrea Cianchi.
\newblock Nonlinear potentials, local solutions to elliptic equations and
  rearrangements.
\newblock {\em Ann. Sc. Norm. Super. Pisa Cl. Sci. (5)}, 10(2):335--361, 2011.

\bibitem[CKS87]{CKS87}
Eric Carlen, Shigeo Kusuoka, and Daniel~W Stroock.
\newblock Upper bounds for symmetric {M}arkov transition functions.
\newblock In {\em Ann. Inst. H. Poincaré. Probab. Statist.}, volume~23, pages
  245--287, 1987.

\bibitem[CMT94]{CMT}
Peter Constantin, Andrew Majda, and Esteban Tabak.
\newblock Formation of strong fronts in the {$2$}-{D} quasigeostrophic thermal
  active scalar.
\newblock {\em Nonlinearity}, 7(6):1495--1533, 1994.

\bibitem[Coz17]{Coz17}
Matteo Cozzi.
\newblock Regularity results and {H}arnack inequalities for minimizers and
  solutions of nonlocal problems: a unified approach via fractional {D}e
  {G}iorgi classes.
\newblock {\em J. Funct. Anal.}, 272(11):4762--4837, 2017.

\bibitem[CV10]{CaffaVasseur}
Luis Caffarelli and Alexis Vasseur.
\newblock Drift diffusion equations with fractional diffusion and the
  quasi-geostrophic equation.
\newblock {\em Ann. of Math. (2)}, 171(3):1903--1930, 2010.

\bibitem[Dav87]{Dav87}
E.~Davies.
\newblock Explicit constants for {G}aussian upper bounds on heat kernels.
\newblock {\em Amer. J. Math.}, 109(2):319--333, 1987.

\bibitem[Dav89]{Dav89}
E.~Davies.
\newblock {\em Heat kernels and spectral theory}.
\newblock Number~92. Cambridge university press, 1989.

\bibitem[DCKP16]{DKP}
Agnese Di~Castro, Tuomo Kuusi, and Giampiero Palatucci.
\newblock Local behavior of fractional {$p$}-minimizers.
\newblock {\em Ann. Inst. H. Poincar\'{e} C Anal. Non Lin\'{e}aire},
  33(5):1279--1299, 2016.

\bibitem[DF22]{DFJMPA}
Cristiana De~Filippis.
\newblock Quasiconvexity and partial regularity via nonlinear potentials.
\newblock {\em J. Math. Pures Appl. (9)}, 163:11--82, 2022.

\bibitem[DiB93]{DiB93}
Emmanuele DiBenedetto.
\newblock {\em Degenerate parabolic equations}.
\newblock Springer Science \& Business Media, 1993.

\bibitem[DM11]{duzaarMin}
Frank Duzaar and Giuseppe Mingione.
\newblock Gradient estimates via non-linear potentials.
\newblock {\em Amer. J. Math.}, 133(4):1093--1149, 2011.

\bibitem[DN23]{DNCZ}
Lars Diening and Simon Nowak.
\newblock Calder\'on-{Z}ygmund estimates for the fractional $p$-{L}aplacian.
\newblock {\em arXiv:2111.05768}, 2023.

\bibitem[DS18]{delgadinoSmith}
Mat\'{\i}as Delgadino and Scott Smith.
\newblock H\"{o}lder estimates for fractional parabolic equations with critical
  divergence free drifts.
\newblock {\em Ann. Inst. H. Poincar\'{e} C Anal. Non Lin\'{e}aire},
  35(3):577--604, 2018.

\bibitem[DZ21]{DongJEMS}
Hongjie Dong and Hanye Zhu.
\newblock Gradient estimates for singular $p$-{L}aplace type equations with
  measure data.
\newblock {\em arXiv:2102.08584; to appear in J. Eur. Math. Soc.}, 2021.

\bibitem[FK13]{FelsKass}
Matthieu Felsinger and Moritz Kassmann.
\newblock Local regularity for parabolic nonlocal operators.
\newblock {\em Comm. Partial Differential Equations}, 38(9):1539--1573, 2013.

\bibitem[Kas09]{KassCalcVar}
Moritz Kassmann.
\newblock A priori estimates for integro-differential operators with measurable
  kernels.
\newblock {\em Calc. Var. Partial Differential Equations}, 34(1):1--21, 2009.

\bibitem[Kis10]{kiselev}
A.~Kiselev.
\newblock Regularity and blow up for active scalars.
\newblock {\em Math. Model. Nat. Phenom.}, 5(4):225--255, 2010.

\bibitem[KM94]{KM94}
Tero Kilpel{\"a}inen and Jan Mal{\'y}.
\newblock The {W}iener test and potential estimates for quasilinear elliptic
  equations.
\newblock {\em Acta Math.}, 172(1):137--161, 1994.

\bibitem[KM14]{KuMi}
Tuomo Kuusi and Giuseppe Mingione.
\newblock Riesz potentials and nonlinear parabolic equations.
\newblock {\em Arch. Ration. Mech. Anal.}, 212(3):727--780, 2014.

\bibitem[KM18]{KuMiV}
Tuomo Kuusi and Giuseppe Mingione.
\newblock Vectorial nonlinear potential theory.
\newblock {\em J. Eur. Math. Soc.}, 20(4):929--1004, 2018.

\bibitem[KMS15]{KMS1}
Tuomo Kuusi, Giuseppe Mingione, and Yannick Sire.
\newblock Nonlocal equations with measure data.
\newblock {\em Comm. Math. Phys.}, 337(3):1317--1368, 2015.

\bibitem[KMS18]{KMS2}
Tuomo Kuusi, Giuseppe Mingione, and Yannick Sire.
\newblock Regularity issues involving the fractional {$p$}-{L}aplacian.
\newblock In {\em Recent developments in nonlocal theory}, pages 303--334. De
  Gruyter, Berlin, 2018.

\bibitem[KN09]{KN}
A.~Kiselev and F.~Nazarov.
\newblock A variation on a theme of {C}affarelli and {V}asseur.
\newblock {\em Zap. Nauchn. Sem. S.-Peterburg. Otdel. Mat. Inst. Steklov.
  (POMI)}, 370(Kraevye Zadachi Matematichesko\u{\i} Fiziki i Smezhnye Voprosy
  Teorii Funktsi\u{\i}. 40):58--72, 220, 2009.

\bibitem[KNS22]{KNS}
Tuomo Kuusi, Simon Nowak, and Yannick Sire.
\newblock Gradient regularity and first-order potential estimates for a class
  of nonlocal equations.
\newblock {\em arXiv:2212.01950}, 2022.

\bibitem[KNV07]{KNV}
A.~Kiselev, F.~Nazarov, and A.~Volberg.
\newblock Global well-posedness for the critical 2{D} dissipative
  quasi-geostrophic equation.
\newblock {\em Invent. Math.}, 167(3):445--453, 2007.

\bibitem[KW22a]{KaWe22a}
Moritz Kassmann and Marvin Weidner.
\newblock Nonlocal operators related to nonsymmetric forms {I}: {H}{\"o}lder
  estimates.
\newblock {\em arXiv:2203.07418}, 2022.

\bibitem[KW22b]{KaWe22b}
Moritz Kassmann and Marvin Weidner.
\newblock Nonlocal operators related to nonsymmetric forms {II}: {H}arnack
  inequalities.
\newblock {\em arXiv:2205.05531, to appear in Anal. PDE}, 2022.

\bibitem[KW23a]{KaWe23}
Moritz Kassmann and Marvin Weidner.
\newblock The parabolic {H}arnack inequality for nonlocal equations.
\newblock {\em arXiv:2303.05975}, 2023.

\bibitem[KW23b]{KaWe23b}
Moritz Kassmann and Marvin Weidner.
\newblock Upper heat kernel estimates for nonlocal operators via {A}ronson’s
  method.
\newblock {\em Calc. Var. Partial Differential Equations}, 62(2):68, 2023.

\bibitem[Lia22]{Lia22}
Naian Liao.
\newblock H{\"o}lder regularity for parabolic fractional p-{L}aplacian.
\newblock {\em arXiv:2205.10111}, 2022.

\bibitem[Min11]{Min}
Giuseppe Mingione.
\newblock Gradient potential estimates.
\newblock {\em J. Eur. Math. Soc.}, 13(2):459--486, 2011.

\bibitem[MM13a]{MaMi13b}
Yasunori Maekawa and Hideyuki Miura.
\newblock On fundamental solutions for non-local parabolic equations with
  divergence free drift.
\newblock {\em Adv. Math.}, 247:123--191, 2013.

\bibitem[MM13b]{MaMi13a}
Yasunori Maekawa and Hideyuki Miura.
\newblock Upper bounds for fundamental solutions to non-local diffusion
  equations with divergence free drift.
\newblock {\em J. Funct. Anal.}, 264(10):2245--2268, 2013.

\bibitem[Ngu14]{QH}
Quoc-Hung Nguyen.
\newblock Potential estimates and quasilinear parabolic equations with measure
  data.
\newblock {\em arXiv:1405.2587, to appear in Mem. Amer. Math. Soc.}, 2014.

\bibitem[NST23]{sireDCDS}
Quoc-Hung Nguyen, Yannick Sire, and Le~Xuan Truong.
\newblock H\"{o}lder continuity of solutions for a class of drift-diffusion
  equations.
\newblock {\em Discrete Contin. Dyn. Syst.}, 43(3-4):1657--1685, 2023.

\bibitem[Sil11]{SilvestreAdv}
Luis Silvestre.
\newblock On the differentiability of the solution to the {H}amilton-{J}acobi
  equation with critical fractional diffusion.
\newblock {\em Adv. Math.}, 226(2):2020--2039, 2011.

\bibitem[Sil12a]{SilvestrePisa}
Luis Silvestre.
\newblock H\"{o}lder estimates for advection fractional-diffusion equations.
\newblock {\em Ann. Sc. Norm. Super. Pisa Cl. Sci. (5)}, 11(4):843--855, 2012.

\bibitem[Sil12b]{SilvestreIndiana}
Luis Silvestre.
\newblock On the differentiability of the solution to an equation with drift
  and fractional diffusion.
\newblock {\em Indiana Univ. Math. J.}, 61(2):557--584, 2012.

\bibitem[TW02]{TWAJM}
Neil Trudinger and Xu-Jia Wang.
\newblock On the weak continuity of elliptic operators and applications to
  potential theory.
\newblock {\em Amer. J. Math.}, 124(2):369--410, 2002.

\end{thebibliography}

\end{document}